\documentclass[a4paper, reqno, 12pt]{amsart}
\usepackage{stmaryrd}
\usepackage{float}

\usepackage[usenames,dvipsnames]{color}
\usepackage{amsthm,amsfonts,amssymb,amsmath,amsxtra}
\usepackage{tikz}
\usepackage{tkz-euclide}
\usepackage[all]{xy}
\SelectTips{cm}{}
\usepackage{xr-hyper}
\usepackage{verbatim}

\usepackage{mathrsfs}

\RequirePackage{xspace}
\RequirePackage{etoolbox}
\RequirePackage{varwidth}
\RequirePackage{enumitem}
\RequirePackage{tensor}
\RequirePackage{mathtools}
\RequirePackage{longtable}
\RequirePackage{multirow}

\def\lup{\rightharpoonup}
\def\ge{\geqslant}
\def\le{\leqslant}
\def\a{\alpha}

\def\g{\gamma}
\def\G{\Gamma}
\def\d{\delta}
\def\D{\Delta}

\def\e{\epsilon}

\def\s{\sigma}
\def\t{\tau}

\def\k{\kappa}
\def\l{\lambda}

\def\i{^{-1}}

\def\<{\langle}
\def\>{\rangle}

\newcommand{\bJ}{\mathbf J}

\newcommand{\bG}{\mathbf G}

\newcommand{\bM}{\mathbf M}

\newcommand{\bP}{\mathbf P}
\newcommand{\bN}{\mathbf N}
\newcommand{\bS}{\mathbf S}
\newcommand{\bT}{\mathbf T}

\newcommand{\Gr}{\mathrm{Gr}}
\newcommand{\Fl}{\mathrm{Fl}}

\newcommand{\rs}{\mathrm{rs}}

\newcommand{\BC}{\ensuremath{\mathbb {C}}\xspace}

\newcommand{\BF}{\ensuremath{\mathbb {F}}\xspace}
\newcommand{{\BG}}{\ensuremath{\mathbb {G}}\xspace}

\newcommand{{\BK}}{\ensuremath{\mathbb {K}}\xspace}

\newcommand{\BM}{\ensuremath{\mathbb {M}}\xspace}
\newcommand{\BN}{\ensuremath{\mathbb {N}}\xspace}

\newcommand{\BQ}{\ensuremath{\mathbb {Q}}\xspace}
\newcommand{\BR}{\ensuremath{\mathbb {R}}\xspace}
\newcommand{\BS}{\ensuremath{\mathbb {S}}\xspace}

\def\bH{\mathbf H}

\newcommand{\CA}{\ensuremath{\mathcal {A}}\xspace}

\newcommand{\CI}{\ensuremath{\mathcal {I}}\xspace}

\newcommand{\CK}{\ensuremath{\mathcal {K}}\xspace}

\newcommand{\CO}{\ensuremath{\mathcal {O}}\xspace}
\newcommand{\CP}{\ensuremath{\mathcal {P}}\xspace}

\newcommand{\CT}{\ensuremath{\mathcal {T}}\xspace}

\newcommand{\Ad}{{\mathrm{Ad}}}
\newcommand{\ad}{{\mathrm{ad}}}

\DeclareMathOperator{\Aut}{Aut}

\newcommand{\af}{\mathrm{af}}

\newcommand{\ex}{{\mathrm{ex}}}

\DeclareMathOperator{\rank}{rank}

\newcommand{\reg}{{\mathrm{reg}}}

\newcommand{\RaP}{{\rm RP}}

\def\tW{\tilde W}
\def\tS{\tilde \BS}
\def\kk{\mathbf k}
\DeclareMathOperator{\supp}{supp}
\newtheorem{theorem}{Theorem}
\newtheorem{proposition}[theorem]{Proposition}
\newtheorem{lemma}[theorem]{Lemma}

\theoremstyle{definition}

\newtheorem{remark}[theorem]{Remark}

\newtheoremstyle{query}%
{}{}%space above/below
{\color{red}}%body style
{}%heading indent
{\sffamily\bfseries}{:}{12pt}%heading style/punctuation/space after
{}% head spec
\theoremstyle{query}
\newtheorem{aq}{Author Query/Comment}

\newcommand{\baq}{\begin{aq}}%This just makes things easier
\newcommand{\eaq}{\end{aq}}

\numberwithin{equation}{section}
\numberwithin{theorem}{section}

\usepackage{lineno}
\setitemize[0]{leftmargin=*,itemsep=\the\smallskipamount}
\setenumerate[0]{leftmargin=*,itemsep=\the\smallskipamount}

\renewcommand{\to}{%
   \ifbool{@display}{\longrightarrow}{\rightarrow}%
   }
\let\shortmapsto\mapsto
\renewcommand{\mapsto}{%
   \ifbool{@display}{\longmapsto}{\shortmapsto}%
   }
\newlength{\olen}
\newlength{\ulen}
\newlength{\xlen}
\newcommand{\xra}[2][]{%
   \ifbool{@display}%
      {\settowidth{\olen}{$\overset{#2}{\longrightarrow}$}%
       \settowidth{\ulen}{$\underset{#1}{\longrightarrow}$}%
       \settowidth{\xlen}{$\xrightarrow[#1]{#2}$}%
       \ifdimgreater{\olen}{\xlen}%
          {\underset{#1}{\overset{#2}{\longrightarrow}}}%
          {\ifdimgreater{\ulen}{\xlen}%
             {\underset{#1}{\overset{#2}{\longrightarrow}}}
             {\xrightarrow[#1]{#2}}}}%
      {\xrightarrow[#1]{#2}}
   }
\makeatother
\newcommand{\xyra}[2][]{%
   \settowidth{\xlen}{$\xrightarrow[#1]{#2}$}%
   \ifbool{@display}%
      {\settowidth{\olen}{$\overset{#2}{\longrightarrow}$}%
       \settowidth{\ulen}{$\underset{#1}{\longrightarrow}$}%
       \ifdimgreater{\olen}{\xlen}%
          {\mathrel{\xymatrix@M=.12ex@C=3.2ex{\ar[r]^-{#2}_-{#1} &}}}%
          {\ifdimgreater{\ulen}{\xlen}%
             {\mathrel{\xymatrix@M=.12ex@C=3.2ex{\ar[r]^-{#2}_-{#1} &}}}
             {\mathrel{\xymatrix@M=.12ex@C=\the\xlen{\ar[r]^-{#2}_-{#1} &}}}}}%
      {\mathrel{\xymatrix@M=.12ex@C=\the\xlen{\ar[r]^-{#2}_-{#1} &}}}%
   }
\makeatletter
\newcommand{\xla}[2][]{%
   \ifbool{@display}%
      {\settowidth{\olen}{$\overset{#2}{\longleftarrow}$}%
       \settowidth{\ulen}{$\underset{#1}{\longleftarrow}$}%
       \settowidth{\xlen}{$\xleftarrow[#1]{#2}$}%
       \ifdimgreater{\olen}{\xlen}%
          {\underset{#1}{\overset{#2}{\longleftarrow}}}%
          {\ifdimgreater{\ulen}{\xlen}%
             {\underset{#1}{\overset{#2}{\longleftarrow}}}
             {\xleftarrow[#1]{#2}}}}%
      {\xleftarrow[#1]{#2}}
   }
\newcommand{\isoarrow}{%
   \ifbool{@display}{\overset{\sim}{\longrightarrow}}{\xrightarrow\sim}%
   }
   
\begin{document}

\title[]{Irreducible components of affine Lusztig varieties}
\author[Xuhua He]{Xuhua He}
\address{Department of Mathematics and New Cornerstone Science Laboratory, The University of Hong Kong, Pokfulam, Hong Kong, Hong Kong SAR, China}
\email{xuhuahe@hku.hk}

\thanks{}


\date{\today}

\begin{abstract}
Let $\breve{G}$ be a loop group and $\tW$ be its Iwahori-Weyl group. The affine Lusztig variety $Y_w(\gamma)$ describes the intersection of the Bruhat cell $\mathcal{I} \dot{w} \mathcal{I}$ for $w \in \tW$ with the conjugacy class of $\gamma \in \breve{G}$, while the affine Deligne-Lusztig variety $X_w(b)$ describes the intersection of the Bruhat cell $\mathcal{I} \dot{w} \mathcal{I}$ with the Frobenius-twisted conjugacy class of $b \in \breve{G}$. Although the geometric connections between these varieties are unknown, numerical relations exist in their geometric properties.

This paper explores the irreducible components of affine Lusztig varieties. The centralizer of $\g$ acts on $Y_w(\g)$ and the Frobenius-twisted centralizer of $b$ acts on $X_w(b)$. We relate the number of orbits on the top-dimensional components of $Y_w(\gamma)$ to the numbers of orbits on top-dimensional components of $X_w(b)$ and the affine Springer fibers. For split groups and elements $\gamma$ with integral Newton points, we show that, for most $w$, the numbers of orbits for the affine Lusztig variety and the associated affine Deligne-Lusztig variety match. Moreover, for these $\g$, we verify Chi's conjecture that the number of top-dimensional components in $Y_\mu(\gamma)$ within the affine Grassmannian equals to the dimension of a specific weight space in a representation of the Langlands dual group.
\end{abstract}

\maketitle

\section*{Introduction}

\subsection{Affine Lusztig varieties and affine Deligne-Lusztig varieties} 
The primary focus of this paper is on the affine analogues of the Lusztig varieties, specifically the affine Lusztig varieties. For simplicity, the introduction is limited to discussing split groups in the equal characteristic case. However, the main body of the paper extends the study to arbitrary reductive groups in both equal and mixed characteristic settings.

Let $L$ be the field of Laurent series over an algebraically closed field and $\bG$ be a connected reductive group split over $L$. Let $\breve G=\bG(L)$ be the corresponding loop group and $\CK=\bG(\CO_L)$ be its hyperspecial subgroup. The affine Lusztig varieties in the affine Grassmannian $\Gr=\breve G/\CK$ and the affine flag variety $\Fl=\breve G/\CI$ are defined by
\begin{gather*} 
Y^{\bG}_\mu(\g)=\{g \CK \in \Gr; g \i \g g \in \CK t^\mu \CK\}, \\
Y_w^{\bG}(\g)=\{g \CI \in \Fl; g \i \g g \in \CI \dot w \CI\}.
\end{gather*}
Here $\g$ is a regular semisimple element in $\breve G$, $\mu$ is a dominant coweight of $\breve G$, and $w$ is an element in the Iwahori--Weyl group of $\breve G$. It was first introduced by Lusztig in~\cite{Lu11}. In the affine Grassmannian case, it was also studied by Kottwitz and Viehmann ~\cite{KV}. In the literature,  affine Lusztig varieties are also referred to as generalized affine Springer fibers or Kottwitz–Viehmann varieties.

Affine Lusztig varieties are locally closed subschemes of the affine Grassmannian and the affine flag variety, equipped with a reduced scheme structure. They naturally arise in the study of orbital integrals of spherical Hecke algebras and Iwahori–Hecke algebras, and they serve as building blocks for the (conjectural) theory of affine character sheaves.

Affine Deligne--Lusztig varieties are defined in a similar way. Let $\bG'$ be a connected reductive group split over $\BF_q(\!(\e)\!)$ and $L'=\overline \BF_q(\!(\e)\!)$. Let $\breve G'=\bG'(L')$ and $\s$ be the Frobenius morphism on $\breve G'$. Let $\CI'$ be a $\s$-stable Iwahori subgroup of $\breve G'$. Let $b \in \breve G'$ and $w$ be an element in the Iwahori--Weyl group of $\breve G'$. The affine Deligne--Lusztig varieties are defined by
\begin{gather*} 
X^{\bG}_\mu(b)=\{g \CK'; g \i b \s(g) \in \CK' t^\mu \CK'\}, \\
X_w^{\bG'}(b)=\{g \CI'; g \i b \s(g) \in \CI' \dot w \CI'\}.
\end{gather*}
Here $b \in \breve G'$, $\mu$ is a dominant coweight of $\breve G'$, and $w$ is an element in the Iwahori--Weyl group of $\breve G'$. The notion of an affine Deligne--Lusztig variety was first introduced by Rapoport ~\cite{Ra} and plays an important role in arithmetic geometry and the Langlands program. 

Despite the similar definitions of affine Deligne–Lusztig varieties and affine Lusztig varieties, there is no direct geometric connection between these two families of varieties. However, the main result of \cite{He-ALV} establishes the following numerical connection between these varieties:

\emph{Suppose that the group $\bG'$ over $F'$ is associated with the group $\bG$ over $L$ and the regular semisimple conjugacy class of $\g$ in $\breve G$ is associated with the $\s$-conjugacy class of $b$ in $\breve G'$. Then, for any $w \in \tW$, }
\begin{itemize}
    \item $Y^{\bG}_w(\g) \neq \emptyset$ if and only if $X^{\bG'}_w(b) \neq \emptyset$;

    \item In this case $\dim Y^{\bG}_w(\g)=\dim X^{\bG'}_w(b)+\dim Y^{\bG}_{\g}$.
\end{itemize}

The notions of associated groups and associated conjugacy classes were introduced in \cite{He-ALV} (see also Theorem \ref{thm:ass}). The affine Springer fiber $Y^{\bG}_{\g}$ associated with $\g$ was defined in \cite{He-ALV} (see also \S\ref{sec:HN}).

%Roughly speaking, the dimension of affine Lusztig varieties equals to the dimension of the affine Deligne-Lusztig varieties plus the dimension of affine Springer fibers. 

\subsection{Main result} 
For a scheme $X$, we denote by $\Sigma^{{\rm top}}(X)$ the set of irreducible components of $X$ of dimension equals to $\dim(X)$. For any regular semisimple element $\g \in \breve G$, its centralizer $Z_{\breve G}(\g)$ acts on $Y^{\bG}_w(\g)$ and on $\Sigma^{{\rm top}}(Y^{\bG}_w(\g))$. Similarly, for any $b \in \breve G'$, the $\s$-centralizer $J_b$ acts on $X^{\bG'}_w(b)$ and on $\Sigma^{{\rm top}}(X^{\bG'}_w(b))$.

The main result of this paper is to enumerate the irreducible components of affine Lusztig varieties. The description is given in terms of reduction paths, a notion introduced in \cite{HNY} to encode the Deligne-Lusztig reduction steps. We refer to \S\ref{sec:tree} for the precise definition. 

We first recall the description of the irreducible components of affine Deligne-Lusztig varieties. 

(a) $\sharp\bigl(J_b \backslash \Sigma^{{\rm top}}(X^{\bG'}_w(b))\bigr)$ equals to the number of reduction paths in a reduction tree of $w$ that correspond to the $\s$-conjugacy class $[b]$ and have the correct lengths. 

We may rewrite it in a slightly different way as 

(a') $\sharp\bigl(J_b \backslash \Sigma^{{\rm top}}(X^{\bG'}_w(b))\bigr)=\sum_{\underline p} 1$, where $\underline p$ runs over the reduction paths in a reduction tree of $w$ that correspond to the $\s$-conjugacy class $[b]$ and have the correct lengths. 

Now we state our main result. 

\begin{theorem}\label{thm:intro1}
    Let $\g$ be a regular semisimple element in $\breve G$. Then for any $w \in \tW$, we have $$\sharp \bigl(Z_{\breve G}(\g) \backslash \Sigma^{{\rm top}}(Y^{\bG}_w(\g))\bigr)=\sum_{\underline p} n_{\underline p, \g},$$ where $\underline p$ runs over the reduction paths in a reduction tree of $w$ that correspond to the conjugacy class $\{\g\}$ and have the correct lengths and $n_{\underline p, \g}$ is the number of $Z_{\breve G}(\g)$-orbits on the irreducible components of affine Springer fibers associated with the pair $(\g, \underline p)$. 
\end{theorem}

We refer to \S\ref{sec:levi} and Theorem \ref{thm:irr-alv} for the definitions and the precise statement. 

\subsection{Comparison theorem on the affine Lusztig varieties and affine Deligne-Lusztig varieties}

It is worth pointing out that the number $n_{\underline p, \g}$ does not equal to $1$ in general. It was shown by Tsai in \cite{Tsai} that for split groups, when $\g$ is ``close enough'' to the identity element, the number of $Z_{\breve G}(\g)$-orbits on the set of top dimensional irreducible components of the affine Springer fibers associated with $\g$ in the affine flag equals to the cardinality of the Weyl group. %For the ramified groups, it was shown by Chi in \cite{Chi2} that the number of $Z_{\breve G}(\g)$-orbits on the set of top dimensional irreducible components of the affine Springer fibers associated with $\g$ in the affine Grassmannian is not $1$ in general. 

However, we discovered that for the split group and the regular semismimple element $\g$ with integral Newton point, when $w$ becomes ``larger'', the difference between the irreducible components of the affine Lusztig varieties and the affine Deligne-Lusztig variety disappears. This leads to our second main result. 

%More precisely, we show that when the translation part of $w$ is regular and a bit larger than the Newton point of $\g$, or when $w$ is in the anti-dominant chamber, then \eqref{eq:main} holds.  

%Another known case is for the ramified groups, it was shown by Chi in ? that the action of the centralizer of $\g$ on the set of irreducible components of the affine Springer fiber in the affine Grassmannian is not even transitive. 

\begin{theorem}\label{thm:main-intro}
%    Suppose that the group $\bG'$ over $F'$ is associated with the group $\bG$ over $L$ and the regular semisimple conjugacy class of $\g$ in $\breve G$ is associated with the $\s$-conjugacy class of $b$ in $\breve G'$. If 

Suppose that $\bG$ is split over $L$ and $\g$ has integral Newton vector. Let $(\bG', [b])$ be the pair associated with $(\bG, \{\g\})$, then for ``almost all'' $w \in \tW$, we have
    \begin{equation}\label{eq:main}
        \sharp\bigl(Z_{\breve G}(\g) \backslash \Sigma^{{\rm top}}(Y^{\bG}_w(\g))\bigr)=\sharp\bigl(J_b \backslash \Sigma^{{\rm top}}(X^{\bG'}_w(b))\bigr).
    \end{equation}
\end{theorem}

Here ``almost all'' means that the exceptional set for the pairs $(w, \{\g\})$ is small with respect to measure, not cardinality. We refer to Theorem \ref{thm:main} for the precise statement. 

As a special case, we deduce the explicit description on the irreducible components of affine Lusztig varieties in the affine Grassmannian for regular semisimple elements with integral Newton vector in split groups, verifying a conjecture of Chi \cite{Chi} for integral Newton points. %We also explain in \S\ref{sec:chi-conj} that how the conjecture fails in general. 

\begin{theorem}
    If $\bG$ is split over $L$ and the Newton vector of $\g$ is integral, then for any dominant coweight $\mu$ such that $Y_\mu(\g) \neq \emptyset$, we have $$\sharp\bigl(Z_{\breve G}(\g) \backslash \Sigma^{{\rm top}}(Y^{\bG}_\mu(\g))\bigr)=\dim V_\mu(\l_\g),$$ where $V_\mu(\l_\g)$ is certain weight space of the irreducible representation $V_\mu$ of the Langlands dual group of $\bG$. 
\end{theorem}

\subsection{Strategy}

We first provide a brief review of the strategy and the main ingredients in \cite{He-ALV} to establish the comparison theorem on the dimension of affine Lusztig varieties and affine Deligne--Lusztig varieties. 

In \cite{He-ALV}, the basic strategy involves two parts:
\begin{itemize}

   \item base case: here $w$ is in the Weyl group of a standard parahoric subgroup $\CK_1$ of $\breve G$. In this case, the affine Lusztig variety $Y^{\bG}_w(\g)$ maps to the affine Springer fiber $Y$ associated with $\g$ in the partial flag variety $\breve G/\CK_1$, and the fibers are the (classical) Lusztig varieties. 
    \begin{itemize}
        \item The affine Springer fiber $Y$ consists of two parts: the regular locus and the non-regular locus;

        \item By the work of Kazhdan and Lusztig \cite{KL} on the regular locus of $Y$ and a detailed analysis on the classical Lusztig varieties, the subvariety of $Y^{\bG}_w(\g)$ over $Y^{\reg}$ has the desired dimension.
    \end{itemize}
    \item Inductive step: We apply Deligne-Lusztig reduction \cite{He-Ann} and Hodge-Newton decomposition \cite{GHKR10}, \cite{He-ALV} to reduce to the case where $w$ is in a(n extended) standard parahoric subgroup of $\tW$.
\end{itemize}

It is worth noting that the reduction paths in the inductive step are rather complicated and that there is no good way to keep track of them.

However, for the affine Lusztig varieties and affine Deligne-Lusztig varieties, they share the same reduction paths. This makes it possible to compare the dimensions of these two families of schemes. 

Following a similar strategy, we prove Theorem \ref{thm:intro1}. The proof of Theorem \ref{thm:main-intro} is more subtle. As mentioned above, the comparison result \eqref{eq:main} is not valid for all $w \in \tW$. Thus one needs to carefully choose the base case to start with. 

\begin{itemize}
\item base case : here $\bG$ is split over $L$ and $\g$ is an element in $\CI$. 

\begin{itemize}
\item By the works of Ngo \cite{Ngo} and Chi \cite{Chi}, \cite{Chi2}, the centralizer $Z_{\breve G}(\g)$ acts transitively on the regular locus of affine Springer fiber associated with $\g$ in the affine Grassmannian;

\item By \cite{HL}, the classical Lusztig varieties associated with the elliptic Weyl group element and regular element in the group is irreducible. 
\end{itemize}

Combining these works, the base case for the irreducible components is where $w$ is an elliptic element in the hyperspecial subgroup of $\tW$. 

\item Inductive step:

%   To extend the result on irreducible components to a larger family of elements in $\tilde{W}$, we combine several deep results and techniques:
     \begin{itemize}
         \item By the work of \cite{HZZ} and \cite{Nie} on the conjecture of X. Zhu and M. Rapoport for the $J_b$-action on the affine Deligne-Lusztig varieties in the affine Grassmannian. These works provide crucial information on the structure of affine Deligne–Lusztig varieties associated with $w_0 t^\mu$ in the affine flag variety.
%         \item Reduction Tree introduced in \cite{HNY}: This method helps in organizing the reduction paths in a systematic manner.
        \item Partial Conjugation Method in \cite{He07}: This technique allows us to handle some part of the reduction paths more effectively and to extend the results from the elements of the form $w_0 t^\mu$ to "almost all" elements in $\tilde{W}$.
     \end{itemize}
Using these strategies and tools, we can establish the main comparison theorem on the irreducible components of affine Lusztig varieties and affine Deligne–Lusztig varieties for a broad class of elements in $\tilde W$.
\end{itemize}

\smallskip

{\bf Acknowledgement: } XH is partially supported by the New Cornerstone Science Foundation through the New Cornerstone Investigator Program and the Xplorer Prize, and by Hong Kong RGC grant 14300122. We thank Jingren Chi, George Lusztig, Michael Rapoport, Cheng-Chiang Tsai and Zhiwei Yun for helpful discussions and suggestions, and Felix Schremmer for careful reading of a preliminary version of this paper. 

\section{Preliminary}
\subsection{Loop groups and  Iwahori--Weyl groups}\label{sec:1.1}
Let $\kk$ be an algebraically closed field and $L=\kk(\!(\e)\!)$ be the field of the Laurent series or $L=W(\kk)[1/p]$ if $\kk$ is characteristic $p$, where $W(\kk)$ is the ring of $p$-typical Witt vectors. Let $\bG$ be a connected reductive group over $L$ and $\breve G=\bG(L)$. Let $\breve G^{\rs}$ be the set of regular semisimple elements in $\breve G$.

We choose a maximal split torus $\bS$ of $\bG$. Let $\bT$ be the centralizer of $\bS$ in $\bG$. Then $\bT$ is a maximal torus of $\bG$. The relative Weyl group of $\breve G$ is defined to be $W_0=\bN(\bT)(L)/\bT(L)$, where $\bN(\bT)$ is the normalizer of $\bT$. The Iwahori--Weyl group of $\breve G$ is defined as $\tW=\bN(\bT)(L)/\bT(L)_0$, where $\bT(L)_0$ is the (unique) parahoric subgroup of $\bT(L)$. 

Consider the apartment $\CA$ of $\bG$ corresponding to $\bS$. Fix an alcove $\mathfrak{a}$ in $\CA$, and let $\CI$ be the associated Iwahori subgroup of $\breve G$.  The simple reflections are reflections along the wall of the base alcove $\mathfrak a$. Let $\tilde \BS$ be the set of simple reflections of $\tW$. By choosing a special vertex $v_0$ in $\mathfrak{a}$, $\tW$ can be expressed as the semidirect product $$\tW=X_*(\bT)_{\G_0} \rtimes W_0=\{x t^\l; \l \in X_*(\bT)_{\G_0}, x \in W_0\},$$ where $X_*(\bT)_{\G_0}$ is the set of coinvariants of the coweight lattice $X_*(\bT)$ under the action of $\G_0=\text{Gal}(\bar L/L)$. The set $\BS = \tilde \BS \cap W_0$ consists of simple reflections of $W_0$. For any $w \in \tW$, we select a representative $\dot{w}$ in $\bN(\bT)(L)$. In this paper, we assume that either $\text{char}(\kk)=0$ or $\text{char}(\kk)>0$ does not divide the order of relative Weyl group $W_0$ of $G$.

The dominant Weyl chamber of $V = X_*(T)_{\G_0} \otimes \BR$ is, by convention, opposite to the unique Weyl chamber containing $\mathfrak{a}$. Let $\D$ denote the set of relative simple roots determined by this dominant Weyl chamber, and $w_0$ the longest element of $W_0$. 

A subset $K$ of $\tilde \BS$ is called {\it spherical} if the subgroup $W_K$ of $\tW$ generated by $s \in K$ is finite. In such cases, $\CP_K$ denotes the standard parahoric subgroup of $\breve G$ generated by $\CI$ and $\dot{w}$ for $w \in W_K$. The subgroup $\breve G_{\af}$ of $\breve G$ is generated by all the parahoric subgroups, and the Iwahori--Weyl group of $\breve G_{\af}$ is denoted by $W_{\af}$. The group $W_{\af}$ is a subgroup of $\tW$ generated by $\tilde \BS$. Let $\Omega$ be the set of length-zero elements in $\tW$. Then, we have the decomposition $\tW = W_{\af} \rtimes \Omega$. %By definition, a parahoric subgroup of $\breve G$ is a finite union of double cosets of an Iwahori subgroup, and any parahoric subgroup of $\breve G$ is conjugate to a standard parahoric subgroup $\CP_K$ for some spherical subset $K$ of $\tilde \BS$.

For any subsets $J$ and $K$ of $\tilde \BS$, we denote by ${}^J \tW$ (resp. $\tW^J$) the set of minimal representatives of $W_J \backslash \tW$ (resp. $\tW / W_J$). The notation ${}^J \tW^K$ stands for ${}^J \tW \cap \tW^K$.

%Let $\CK=\bG(\CO_L)$ be the hyperspecial subgroup of $\breve G$ and $\CI \subset \CK$ be the Iwahori subgroup determined by $\mathbf B$. A subset $K$ of $\tilde \BS$ is called {\it spherical} if the subgroup $W_K$ of $\tW$ generated by $s \in K$ is finite. In this case, we denote by $\CP_K$ the standard parahoric subgroup of $\breve G$ generated by $\CI$ and $\dot w$ for $w \in W_K$. By definition, a parahoric subgroup of $\breve G$ is a finite union of double cosets of an Iwahori subgroup. Any parahoric subgroup of $\breve G$ is conjugate to a standard parahoric subgroup $\CP_K$ for some spherical subset $K$ of $\tilde \BS$. 

%For any $J \subset \tilde \BS$, we denote by ${}^J \tW$ (resp. $\tW^J$) the set of minimal representatives of $W_J \backslash \tW$ (resp. $\tW/W_J$). For any $J, K \subset \tilde \BS$, we simply write ${}^J \tW^K$ for ${}^J \tW \cap \tW^K$. 

%Let $\breve G_{\af}$ be the subgroup of $\breve G$ generated by all the parahoric subgroups and $W_{\af}$ be the Iwahori--Weyl group of $\breve G_{\af}$. Then $W_{\af}$ is the subgroup of $\tW$ generated by $\tilde S$. Let $\Omega$ be the set of length-zero elements in $\tW$. Then we have $\tW=W_{\af} \rtimes \Omega$. 

\subsection{The map from $B(\bG)$ to $B_\s(\bG')$}\label{sec:ass}
For any $\g \in \breve G$, let $\{\g\}=\{g \g g \i; g \in \breve G\}$ be the conjugacy class of $\g$. Let $B(\bG)$ be the set of conjugacy classes of $\breve G$. 

Let $F' = \mathbb{F}_q(\!(\e)\!)$ and $\breve F' = \overline{\mathbb{F}}_q(\!(\e)\!)$. Consider a connected reductive group $\bG'$ that is residually split over $F'$, with $\breve G' = \bG'(\breve F')$. Let $\s$ denote the Frobenius morphism of $\breve F'$ over $F'$. We use the same symbol $\s$ for the induced Frobenius endomorphism on $\breve G'$. We fix a $\s$-stable Iwahori subgroup $\CI'$ of $\breve G'$. The Iwahori--Weyl group of $\breve G'$ is denoted by $\tW'$, and the relative Weyl group is denoted by $W'_0$. Let $\G'_0$ be the absolute Galois group of $F'$. We define $X_*(\bT')_\BQ^+$ in a similar way as in \S\ref{sec:1.1}. For any $b \in \breve G'$, let $[b]=\{g b \s(g) \i; g \in \breve G'\}$ be the $\s$-conjugacy class of $b$. Let $B_\s(\bG')$ be the set of $\s$-conjugacy classes of $\breve G'$.  
 
%Note that the isomorphism classes of affine Deligne--Lusztig varieties depend on the $\s$-conjugacy class $[b]$, not on the representative $b$ within $[b]$. Similarly, the isomorphism classes of  affine Lusztig varieties depend on the conjugacy class $\{\g\}$, not the representative $\g$ within $\{\g\}$. 

We recall Kottwitz's classification of $B_\s(\bG')$ in ~\cite{kottwitz-isoI,kottwitz-isoII} and its analog on $B(\bG)$ obtained in  ~\cite{KV} and ~\cite{HN20}. 

\begin{theorem}\label{thm:ass}
(1) There is an embedding $$f_\s: B_\s(\bG') \to \pi_1(\bG')_{\G_0'} \times X_*(\bT')_\BQ^+, \quad [b] \mapsto (\k([b]), \nu_{[b]}).$$ 

(2) There is a natural map $$f: B(\bG) \to \pi_1(\bG)_{\G_0} \times X_*(\bT)_\BQ^+, \quad \{\g\} \mapsto (\k(\{\g\}), \nu_{\{\g\}}).$$
\end{theorem}

We will not give the precise definition of the maps $f$ and $f_\s$ here, but we will provide explicit descriptions of their restrictions to the Iwahori--Weyl groups.

Let $w \in \tW$. There exists a positive integer $n$ such that $w^n=t^\l$ for some $\l \in X_*(\bT)_{\G_0}$. Set $\nu_w=\frac{\l}{n} \in X_*(\bT)_{\BQ}$. The element $\nu_w$ is independent of the choice of $n$. Let $\bar \nu_w$ be the unique dominant coweight in the $W_0$-orbit of $\nu_w$. We also have a natural identification $\tW/W_{\af} \cong \Omega \cong \pi_1(\bG)_{\G_0}$. We consider the maps $f_{\tilde W}: \tW \to \pi_1(\bG)_{\G_0} \times X_*(\bT)_\BQ^+$ given by $w \mapsto (w W_{\af}, \bar \nu_w)$. 
By definition, an element $w \in \tW$ is called {\it straight} if $\ell(w^n)=n \ell(w)$ for all $n \in \BN$. A conjugacy class of $\tW$ is called {\it straight} if it contains a straight element. We denote by $\tW \sslash \tW$ the set of straight conjugacy classes of $\tW$. By~\cite[Theorem 3.3]{HN14}, this map induces an injection $\tW \sslash \tW \to \pi_1(\bG)_{\G_0} \times X_*(\bT)_\BQ^+$. Similarly, we have an injection $\tW' \sslash \tW' \to \pi_1(\bG')_{\G'_0} \times X_*(\bT')_\BQ^+$.

We say that the group $\bG'$ over $F'$ is {\it associated with} the group $\bG$ over $L$ if we have a length-preserving isomorphism $\tW \cong \tW'$ that is compatible with the semidirect products, that is
\[
\xymatrix{X_*(T)_{\G_0} \rtimes W_0 \ar[d]^{\cong} \ar@{=}[r] & \tW \ar[d]^{\cong} \\ X_*(T')_{\G'_0} \rtimes W'_0  \ar@{=}[r] & \tW'
}
\]
In particular, we require that this isomorphism induces isomorphisms on the coweight lattices and on the relative Weyl groups and a bijection on the set of simple reflections. By the classification of reductive groups, for any $\bG$, there exists a connected reductive group $\bG'$ over $F'$ that is associated with it. 

Suppose that $\bG'$ is associated with $\bG$. Then we may identify $\tW'$ with $\tW$. Under this identification, we have $\tW \sslash \tW=\tW' \sslash \tW'$. By \cite[\S 1.6]{He-ALV}, there exists a map  $f_{\bG, \bG'}: B(\bG) \to B_\s(\bG')$ such that the following diagram commutes:
\[
\xymatrix{
B(\bG) \ar@{->>}[r]^{f_{\bG, \bG'}} \ar@{->>}[d] & B_\s(\bG') \ar[d] \\
\tW \sslash \tW \ar@{=}[r] & \tW' \sslash \tW' \ar[u]}
\]
Note that $B_\s(\bG')$ is in natural bijection with $\tW \sslash \tW$. Thus the map $f_{\bG, \bG'}: B(\bG) \to B_\s(\bG')$ is surjective. However, this map is not injective in general. 

\subsection{Affine Lusztig Varieties and Affine Deligne--Lusztig varieties}\label{sec:ADLV}

The affine flag variety $\Fl$ of the group $\breve G$ is defined as $\Fl = \breve G / \CI$, where $\CI$ denotes an Iwahori subgroup. When considering the maximal special parahoric subgroup $\CK^{\mathrm{sp}} \supset \CI$ corresponding to a special vertex $v_0$, we also define the affine Grassmannian as $\Gr = \breve G / \CK^{\mathrm{sp}}$. For $L$ being the field of Laurent series, both the affine flag variety and the affine Grassmannian possess natural ind-scheme structures. In a mixed characteristic context, these varieties are viewed as perfect ind-schemes according to Zhu~\cite{Zhu} and Bhatt–Scholze~\cite{BS}.

Let $X_*(T)_{\G_0}^+$ denote the set of dominant coweights. The group $\breve G$ can be decomposed into disjoint unions: $$\breve G = \bigsqcup_{w \in \tW} \CI \dot w \CI, \qquad \breve G = \bigsqcup_{\mu \in X_*(T)_{\G_0}^+} \CK^{\mathrm{sp}} t^\mu \CK^{\mathrm{sp}}.$$

For any $\g \in \breve G$ and $w \in \tW$, the {\it affine Lusztig variety} $Y_w(\g)$ in the affine flag variety $\Fl$ is defined by $$Y_w(\g)=\{g \CI \in \Fl; g \i \g g \in \CI \dot w \CI\}.$$ For $w \in \tW$ with $\ell(w)=0$ and $\g \in \CI \dot w \CI$, we have $\CI \dot w \CI=\CI \g$, and $Y_w(\g)=\Fl^{\g}$ is the affine Springer fiber.

For any dominant coweight $\mu$ and $\g \in \breve G$, the {\it affine Lusztig variety} $Y_\mu(\g)$ in the affine Grassmannian $\Gr$ is defined by $$Y_\mu(\g)=\{g \CK^{\mathrm{sp}} \in \Gr; g \i \g g \in \CK^{\mathrm{sp}} t^\mu \CK^{\mathrm{sp}}\}.$$

%For any $\g \in \breve G$ and $w \in \tW$, the \textit{affine Lusztig variety} $Y_w(\g)$ within the affine flag variety $\Fl$ is defined as: $$Y_w(\g) = { g \CI \in \Fl ;|; g^{-1} \gamma g \in \CI \dot w \CI }.$$ When $w \in \tW$ satisfies $\ell(w) = 0$ and $\gamma \in \CI \dot w \CI$, we have $\CI \dot w \CI = \CI \gamma$, and thus $Y_w(\gamma) = \Fl^{\gamma}$, which is known as the affine Springer fiber.

%Similarly, for any dominant coweight $\mu$ and $\gamma \in \breve G$, the \textit{affine Lusztig variety} $Y_\mu(\g)$ within the affine Grassmannian $\Gr$ is defined as: $$Y_\mu(\g) = { g \CK^{\mathrm{sp}} \in \Gr ;|; g^{-1} \gamma g \in \CK^{\mathrm{sp}} t^\mu \CK^{\mathrm{sp}} }.$$

By definition, $Y_w(\g)$ forms a locally closed ind-subscheme of $\Fl$, and $Y_\mu(\g)$ forms a locally closed ind-subscheme of $\Gr$. The centralizer of $\g$ in $\breve G$, denoted by $Z_{\breve G}(\g)=\{g \in \breve G; g \i \g g=\g\}$, acts on both the affine Lusztig varieties $Y_w(\g)$ and $Y_\mu(\g)$.  

Note that the isomorphism classes of affine Lusztig varieties depend on the conjugacy class $\{\g\}$, not the representative $\g$ within $\{\g\}$. 

%\subsection{Affine Deligne--Lusztig varieties}

%Let $F'=\mathbb F_q((\e))$ and $\breve F'=\overline{\mathbb F}_q((\e))$. Let $\bG'$ be a connected reductive group, split over $F'$ and $\breve G'=\bG'(\breve F')$. Let $\s$ be the Frobenius morphism on $\breve G'$. Let $\CK'=\bG(\CO_{\breve F'})$ be the hyperspecial subgroup of $\breve G'$. We fix a $\s$-stable Iwahori--Weyl group $\CI' \subset \CK'$. Let $\tW'$ be the Iwahori--Weyl group of $\breve G'$ and $W'_0$ be the relative Weyl group. 

Similarly, for $b \in \breve G'$ and $w \in \tW'$, the affine Deligne--Lusztig variety $X_w(b)$ in the affine flag variety is defined by $$X_w(b)=\{g \CI\in \Fl'; g \i b \s(g) \in \CI' \dot w \CI'\}$$ and the affine Deligne--Lusztig variety in the affine Grassmannian $\Gr'=\breve G'/\CK'^{, \mathrm{sp}}$ defined by $$X_\mu(b)=\{g \CK'^{, \mathrm{sp}} \in \Gr'; g \i b \s(g) \in \CK'^{, \mathrm{sp}} t^\mu \CK'^{, \mathrm{sp}}\}.$$

It is known that  affine Deligne--Lusztig varieties are subschemes of locally finite type in the affine flag variety (resp. affine Grassmannian). Let $\bJ_b$ be the $\s$-centralizer of $b$ and $J_b=\{g \in \breve G'; g \i b \s(g)=b\}$ be the group of $F'$-rational points of $\bJ_b$. Then $J_b$ acts naturally on the affine Deligne--Lusztig varieties associated with $b$. 

Note that the isomorphism classes of affine Deligne--Lusztig varieties depend on the $\s$-conjugacy class $[b]$, not on the representative $b$ within $[b]$.

%By understanding these maps and their properties, we gain insights into the structure of affine Lusztig and Deligne--Lusztig varieties, as well as their dependence on conjugacy and \(\sigma\)-conjugacy classes. This relationship helps facilitate the translation of problems and results between different groups \(\bG\) and \(\bG'\) that are associated with each other.

%In summary, the maps \(f_\sigma\) and \(f\) provide a framework for classifying conjugacy and \(\sigma\)-conjugacy classes in terms of fundamental group data and coweight lattices. The identification of associated groups \(\bG\) and \(\bG'\) through isomorphisms of their Iwahori--Weyl groups further allows us to bridge the gap between the affine structures of these groups.

\subsection{Hodge-Newton decompositions}\label{sec:HN}
For $w \in \tW$, let $\bP_{\nu_w}$ be the parabolic subgroup of $\bG$ generated by $\bT$ and the root subgroups $\mathbf U_\a$ with $\<\a, \nu_w\> \ge 0$, $\bM_{\nu_w}$ be the Levi subgroup of $\bP_{\nu_w}$ generated by $\bT$ and the root subgroups $\mathbf U_\a$ with $\<\a, \nu_w\>=0$. Let $\breve M_{\nu_w}=\bM_{\nu_w}(\breve F)$ and $\CI_{\bM_{\nu_w}}=\CI \cap \breve M_{\nu_w}$ be an Iwahori subgroup of $\breve M_{\nu_w}$. Let $\tW_{\nu_w}$ be the Iwahori--Weyl group of $\breve M_{\nu_w}$. We have the following Hodge-Newton decompositions for the affine Deligne--Lusztig varieties and affine Lusztig varieties.

%The Hodge--Newton decomposition for affine Deligne--Lusztig varieties was established in~\cite[Theorem 2.1.4]{GHKR10} for split groups and in~\cite[Theorem 3.3.1]{GHN} in general. We state the following special case, which will be used in this paper. 

\begin{theorem}~\cite[Theorem 2.1.4]{GHKR10}\label{thm:HN}
Suppose that $w$ is a minimal length element in its conjugacy class in $\tW$. Then $X_w(b) \neq \emptyset$ implies that $[b] \cap \bM'_{\nu_w}(\breve F') \neq \emptyset$. Moreover, for $b \in \bM'_{\nu_w}(\breve F')$, we have $\bJ_b^{\bM'_{\nu_w}}(F')=\bJ_b^{\bG'}(F')$, where $\bJ_b^{\bM'_{\nu_w}}(F')$ is the $\s$-centralizer of $b$ in $\bM'(\breve F)$, and $$X^{\bM'_{\nu_w}}_w(b) \cong X^{\bG'}_w(b).$$
\end{theorem}

%The following Hodge-Newton decomposition of the affine Lusztig varieties was established in \cite[Theorem 3.4]{He-ALV}. 

\begin{theorem}\cite[Theorem 3.4]{He-ALV}\label{thm:HN1}
Suppose that $w$ is a minimal length element in its conjugacy class in $\tW$. Then, for any $\g \in \breve G^{\rs}$, $Y_w(\g) \neq \emptyset$ if and only if $\{\g\} \cap \CI_{\bM_{\nu_w}} \dot w \CI_{\bM_{\nu_w}} \neq \emptyset$. Moreover, for $\g \in \CI_{\bM_{\nu_w}} \dot w \CI_{\bM_{\nu_w}}$, we have $Z_{\breve M_{\nu_w}}(\g)=Z_{\breve G}(\g)$ and  $$Y^{\bM_{\nu_w}}_w(\g) \cong Y^{\bG}_w(\g).$$
\end{theorem}

Let $\g \in \breve G^{\rs}$ and $C$ be the straight conjugacy class of $\tW$ associated with $\{\g\}$. Let $w, w'$ be minimal length elements of $C$. We set $Y_{\g}=Y_w(\g)$ for some minimal length element $w$ of $C$. By \cite[\S 5.3]{He-ALV}, the isomorphism class of $Y_{\g}$ is independent of the choice of such a minimal length representative and $Y_\g \cong Y^{\bG}_w(\g) \cong Y^{\bM_{\nu_w}}_w(\g)$ is the affine Springer fiber of $\g$ in $\breve M_{\nu_w}$. We call it the {\it affine Springer fiber associated with} $\{\g\}$.

\subsection{Dimension formula}\label{sec:affineS} 
%When $\breve F=\BC((\e))$, the dimension formula for affine Springer fibers was established by Bezrukavnikov~\cite{Be} for split groups and extended to ramified groups by Oblomkov and Yun~\cite{OY}. The same result holds for $\breve F=\kk((\e))$ when the characteristic of $\kk$ is large. Under this assumption on $\kk$, we have $$\dim Y_\mu(\g)=\<\mu, \rho\>+\frac{1}{2}(d(\g)-c(\g)),$$ where $d(\g)$ is the discriminant valuation of $\g$ (see~\cite[Definition 3.1.2]{Chi}), and $c(\g)=\rank_{\breve F}(\bG)-\rank_{\breve F}(\bG_\g)$. 

In \cite[Theorem 5.6]{He-ALV}, we established the following comparison result on the dimension formula of affine Lusztig varieties and affine Deligne-Lusztig varieties. 

\begin{theorem}\label{thm:le}
Suppose that $\bG'$ is associated with $\bG$. Let $\g$ be a regular semisimple element of $\breve G$ and $[b]=f_{\bG, \bG'}(\{\g\}) \in B_\s(\bG')$. Then, for any $w \in \tW$, we have 
\begin{enumerate}
    \item $Y_w(\g) \neq \emptyset$ if and only if $X_w(b) \neq \emptyset$; 

    \item If $X_w(b) \neq \emptyset$, then $\dim Y_w(\g)=\dim X_w(b)+\dim Y_{\g}.$
\end{enumerate}

\end{theorem}

%\subsection{Some known results on affine Deligne-Lusztig varieties}
Combining Theorem \ref{thm:le} with our previous knowledge on the affine Deligne-Lusztig varieties, we obtain the explicit description on the nonemptiness pattern and dimension formula of the affine Lusztig varieties in most cases. We first recall the definitions of the defect and virtual dimension. 

For $b \in \breve G'$, the {\it defect} of $b$ is defined as $\operatorname{def}(b)=\rank_F \bG'-\rank_F \bJ_b$, where for a reductive group $\bH$ defined over $F$, $\rank_F$ is the $F$-rank of the group $\bH$. Note that any element in $\tW$ can be written in a unique way as $x t^\l y$, where $x, y \in W_0$, $\l$ is a dominant coweight such that $t^\l y \in {}^{\BS_0} \tW$. We set $\eta(x t^\l y)=y x \in W_0$. For any $w \in \tW$ and $b \in \breve G'$, the {\it virtual dimension} is defined as $$d_w([b])=\frac 12 \big( \ell(w) + \ell(\eta(w)) -\operatorname{def}(b)  \big)-\<\nu_{[b]}, \rho\>.$$ Here $\rho$ is the half sum of positive roots. By~\cite[Theorem 2.30]{He-CDM}, we have \[\tag{a} \dim X_w(b) \le d_w([b]).\] For $\g \in \breve G^{\rs}$, we set $d_w(\{\g\})=d_w([b])$, where $[b]=f_{\bG, \bG'}(\{\g\})$. By Theorem \ref{thm:le}, we have \[\tag{b} \dim Y_w(\g) \le d_w(\{\g\})+\dim Y_\g.\] 

%We may reformulate the virtual dimension using the straight conjugacy class $C$ instead. Set $$d_w(C)=\frac 12 \big( \ell(w) + \ell(\eta(w))-\operatorname{def}(C)-\ell(C)\big).$$ Here $\ell(C)$ is defined as the length of any minimal length element in $C$. We have $\ell(C)=\<\nu_{[b]}, 2 \rho\>$ if $[b]=[\dot w]$ for any $w \in C$. 

%For affine Deligne--Lusztig varieties in the affine Grassmannian, the nonemptiness pattern was obtained by Rapoport and Richartz~\cite{RR} Gashi~\cite{Ga}, and dimension formula was proved by G\"ortz, Haines, Kottwitz, Reuman ~\cite{GHKR06} and Viehmann~\cite{Vi06}  for split groups, by Hamacher~\cite{Ham15} and Zhu~\cite{Zhu} for unramified groups. 

%\begin{theorem}\label{thm:gr}
%Let $\mu$ be a dominant coweight. Then $X_\mu(b) \neq \emptyset$ if and only if $\k(\mu)=\k(b)$ and $\nu_{[b]} \le \mu$. In this case, $$\dim X_\mu(b)=\<\mu-\nu_{[b]}, \rho\>-\frac{1}{2}\operatorname{def}(b).$$ 
%\end{theorem}

%Combining Theorem \ref{thm:gr} with Theorem \ref{thm:le}, we have

%\begin{theorem}\label{thm:gr'}
%Let $\mu$ be a dominant coweight and $\g \in \breve G^{\rs}$. Then $Y_\mu(\g) \neq \emptyset$ if and only if $\k(\mu)=\k(\g)$ and $\nu_{\{\g\}} \le \mu$. In this case, $$\dim Y_\mu(\g)=\<\mu-\nu_{\{\g\}}, \rho\>-\frac{1}{2}\operatorname{def}(b).$$ 
%\end{theorem}

%We do not have a complete description of the nonemptiness pattern and dimension formula for  affine Deligne--Lusztig varieties in the affine flag variety. However, ~\cite[Theorem 6.1]{He20} provides the answer for almost all cases. 
We have the following results on the nonemptiness pattern and dimension formula of affine Deligne-Lusztig varieties and affine Lusztig varieties for almost all cases. 

\begin{theorem}\label{thm:gasf-fl}\cite[Theorem 6.1]{He20}
Suppose that $\<\mu, \a\> \ge 2$ for any simple root $\a$ and $\nu_{[b]}+2 \rho^\vee \le \mu$. Then, for any $x, y \in W_0$, $X_{x t^\mu y}(b) \neq \emptyset$ if and only if $\k(b)=\k(\mu)$ and $\supp(y x)=\BS$. In this case, $$\dim 
X_{x t^\mu y}(b)=d_{x t^\mu y}([b]).$$
\end{theorem}

\begin{theorem}\label{thm:gasf-fl'}\cite[Theorem 6.9]{He-ALV}
Let $\g \in \breve G^{\rs}$. Suppose that $\<\mu, \a\> \ge 2$ for any simple root $\a$ and $\nu_{\{\g\}}+2 \rho^\vee \le \mu$. Then, for any $x, y \in W_0$, $Y_{x t^\mu y}(\g) \neq \emptyset$ if and only if $\k(\g)=\k(\mu)$ and $\supp(y x)=\BS$. In this case, $$\dim 
Y_{x t^\mu y}(\g)=d_{x t^\mu y}(\{\g\})+\dim Y_\g.$$
\end{theorem}

\section{Enumeration of irreducible components}

\subsection{Minimal length elements}
For any conjugacy class $C$ in $\tW$, we denote by $C_{\min}$ the set of minimal length elements in $C$. For $w, w' \in \tW$ and $s \in \tS$, we write $w \xrightarrow{s} w'$ if $w'=s w s$ and $\ell(w') \le \ell(w)$. We write $w \to w'$ if there is a sequence $w=w_0, w_1, \dots, w_n=w'$ of elements in $\tW$ such that for any $k$, $w_{k-1} \xrightarrow{s_k} w_k$ for some $s_k \in \tS$. We write $w \approx w'$ if $w \to w'$ and $w' \to w$. It is easy to see that $w \approx w'$ if $w \to w'$ and $\ell(w)=\ell(w')$. We have the following results. 

\begin{theorem}\label{thm:min}\cite[Theorem 2.10]{HN14}
	Let $C$ be a conjugacy class of $\tW$ and $w \in C$. Then there exists $w' \in C_{\min}$ such that $w \to w'$.
\end{theorem}

We have the following results on nice representatives of minimal length elements.

\begin{theorem}\label{thm:ux}\cite[Propositions 2.4 and 2.7]{HN14}
Let $w \in \tW$. Then there exists a straight element $x \in \tW$, $K \subset \tilde \BS$ with $W_K$ finite, $x \in {}^K \tW^K$, and $\Ad(x)(K)=K$ and $u \in W_K$ such that $u x$ is a minimal length element in the conjugacy class of $w$ and $w \to u x$. 
\end{theorem}

We also point out that in this case, $u$ is a minimal length element in the $\Ad(x)$-conjugacy class of $W_K$.

%\begin{theorem}\label{thm:str-cyc}\cite[Theorem 3.8]{HN14}
%Let $C$ be a straight conjugacy class of $\tW$. Then, for any straight elements $w$ and $w'$ in $C$, we have $w \approx w'$. 
%\end{theorem}

\subsection{Reduction trees}\label{sec:tree}
The reduction trees for $w \in \tW$ was introduced in \cite{HNY}. The reduction tree encodes the Deligne-Lusztig reduction for the affine Deligne-Lusztig varieties associated to $w$ (and to all $b \in \breve G$). The vertices of the graphs are the elements of $\tW$ and the (oriented) edges are of the form $x \lup y$. The reduction trees are constructed inductively. 

For $x, y \in \tW$, we write $x \lup y$ if there exists $x' \in \tW$ and $i \in \tilde \BS$ with $x \approx_\s x'$, $s_i x' \s(s_i)<x'$ and $y \in \{s_i x', s_i x' \s(s_i)\}$. 
The reduction tree of a minimal length in its conjugacy class consists of the element itself and no edges. Suppose that $w$ is not of minimal length in its $\s$-conjugacy class of $\tW$ and that a reduction tree is given for any $z \in \tW$ with $\ell(z)<\ell(w)$. By Theorem \ref{thm:min}, there exists $w' \in \tW$ and $i \in \tilde \BS$ with $w \approx_\s w'$ and $s_i w' \s(s_i)<w'$. The reduction tree of $w$ is the graph containing the given reduction tree for $s_i w'$ and the reduction tree for $s_i w' \s(s_i)$, and the edges $w \lup s_i w'$ and $w \lup s_i w' \s(s_i)$. 

An {\it end point} of a reduction tree $\CT$ is a vertex $x$ of $\CT$ such that there is no edge of the form $x \lup x'$ in $\CT$. A {\it reduction path} in $\CT$ is a path $\underline p: w \lup w_1 \lup \cdots \lup w_n$, where $w_n$ is an end point of $\CT$. We write ${\rm end}(\underline p)=w_n$. The {\it length} $\ell(\underline p)$ of the reduction path $\underline p$ is the number of edges in $\underline p$. %and $[b]_{\underline p}=\Psi({\rm end}(\underline p)) \in B(\bG)$. \remind{define $\{\g\}_{\underline p}$}

We denote by $\RaP(\CT)$ the set of reduction paths in $\CT$. For any $[b] \in B_\s(\bG')$, we denote by $\RaP_{[b]}(\CT)$ the set of reduction paths $\underline p$ in $\CT$ with $f_{\tW}({\rm end}(\underline p))=f_{\s}([b])$. Similarly, for any $\{\g\} \in B(\bG)$, we denote by $\RaP_{\{\g\}}(\CT)$ the set of reduction paths $\underline p$ in $\CT$ with $f_{\tW}({\rm end}(\underline p))=f(\{\g\})$.

The following connection between the reduction trees and the affine Deligne-Lusztig varieties is given in \cite[Proposition 3.9]{HNY}. 

\begin{proposition}\label{prop:dec}
Let $w \in \tW$ and $\CT$ be a reduction tree of $w$. Then for any $b \in \breve G'$, there exists a decomposition $$X_w(b)=\bigsqcup_{\underline p \in \RaP_{[b]}(\CT)} X_{\underline p} (b),$$ where $X_{\underline p}$ is a locally closed subscheme of $X_w(b)$, and is $J_b$-equivariant universally homeomorphic to a fibration over  $X_{\text{end}(\underline p)}(b)$ with irreducible fibers of dimension $\ell(\underline p)$.
\end{proposition}

One may prove in the same way that 

\begin{proposition}\label{prop:dec2}
Let $w \in \tW$ and $\CT$ be a reduction tree of $w$. Then for any $\g \in \breve G$, there exists a decomposition $$Y_w(\g)=\bigsqcup_{\underline p \in \RaP_{\{\g\}}(\CT)} Y_{\underline p}(\g),$$ where $Y_{\underline p}$ is a locally closed subscheme of $Y_w(\g)$, and is $Z_{\breve G}(\g)$-equivariant isomorphic to a fibration over  $Y_{\text{end}(\underline p)}(\g)$ with irreducible fibers of dimension $\ell(\underline p)$.
\end{proposition}

\subsection{Enumeration of irreducible components of affine Deligne-Lusztig varieties}

We first give a description of the irreducible components of affine Deligne-Lusztig varieties in terms of reduction paths. 

\begin{proposition}\label{prop:irr-adlv}
    Let $w \in \tW$ and $\CT$ be a reduction tree of $\CT$. Then for any $b \in \breve G'$ with $X_w(b) \neq \emptyset$, we have $$\sharp\bigl(J_b \backslash \Sigma^{{\rm top}}(X_w(b))\bigr)=\sharp\{\underline p \in \RaP_{[b]}(\CT); \ell(\underline p)+\ell(\text{end}(\underline p))-\<\bar \nu_{\text{end}(\underline p)}, 2 \rho\>=\dim X_w(b)\}. $$
\end{proposition}

\begin{remark}
    It is shown in \cite[Theorem 2.19]{He-CDM} that the number of $J_b$-orbits on $\Sigma^{{\rm top}}(X^{\bG'}_w(b)))$ equals the leading coefficient of the corresponding class polynomial of the associated affine Hecke algebra. One may deduce the above formula from loc.cit. For the convenience of the readers, below we give a proof without referring to the class polynomials. 
\end{remark}

\begin{proof}
    We first consider the case where $w$ is a minimal length element in its conjugacy class in $\tilde W$. In this case, by \cite[Theorem 4.8]{He-Ann}, $\dim X_w(b)=\ell(w)-\<\bar \nu_w, 2 \rho\>$ and $J_b$ acts transitively on the set of irreducible components of $X_w(b)$. Hence $\sharp\bigl(J_b \backslash \Sigma^{{\rm top}}(X_w(b))\bigr)=1$. Also by definition, $\CT$ consists of only one vertex, the element $w$ itself. The statement holds in this case. 

    Now we consider the general case. Let $d=\dim X_w(b)$. By Proposition \ref{prop:dec}, we have \begin{align*} \Sigma^{{\rm top}}(X_w(b)) &=\bigsqcup_{\underline p \in \RaP_{[b]}(\CT); \ell(\underline p)+\dim X_{\text{end}(\underline p)} (b)=d} \Sigma^{{\rm top}}(X_{(\underline p)} (b)) \\ & \cong \bigsqcup_{\underline p \in \RaP_{[b]}(\CT); \ell(\underline p)+\dim X_{\text{end}(\underline p)} (b)=d} \Sigma^{{\rm top}}(X_{\text{end}(\underline p)} (b))
    \end{align*}

    Note that for any $\underline p$, $\text{end}(\underline p)$ is a minimal length element in its conjugacy class in $\tilde W$. Thus $\sharp\bigl(J_b \backslash \Sigma^{{\rm top}}(X_{\text{end}(\underline p)}(b))\bigr)=1$ and $\ell(\underline p)+\dim X_{\text{end}(\underline p)} (b)=\ell(\underline p)+\ell(\text{end}(\underline p))-\<\bar \nu_{\text{end}(\underline p)}, 2 \rho\>$. So $$\sharp\bigl(J_b \backslash \Sigma^{{\rm top}}(X_w(b))\bigr)=\sharp\{\underline p \in \RaP_{[b]}(\CT); \ell(\underline p)+\ell(\text{end}(\underline p))-\<\bar \nu_{\text{end}(\underline p)}, 2 \rho\>=\dim X_w(b)\}. $$ The statement is proved. 
\end{proof}

\subsection{Regular locus}\label{sec:regular}
In this subsection, we study the affine Springer fibers and their regular loci. This is a key ingredient to our description of irreducible components of affine Lusztig varieties. 

We consider the group $\breve G$ together with a group automorphism $\d$ such that $\d(\CI)=\CI$. Let $\breve G_{\ad}$ be the adjoint group of $\breve G$. We assume furthermore that the image of $\d$ in $\Aut(\breve G_{\ad})$ is of finite order. The automorphism $\d$ on $\breve G$ induces group automorphisms on $\tW$ and on $\Omega$. We consider the group extensions $$\breve G^{\ex}=\breve G \rtimes \<\d\>, \tW^{\ex}=\tW \rtimes \<\d\> \text{ and } \Omega^{\ex}=\Omega \rtimes \<\d\>.$$ 

%An element of $\breve G^{\ex}$ is called {\it bounded} if it is contained in a bounded subgroup. An element of $\breve G^{\ex}$ is called {\it bounded-modulo-center} if its image in $\Aut(\breve G_{\ad})$ is bounded. By \cite[\S 2.1]{He-ALV}, an element of $\breve G^{\ex}$ is bounded-modulo-center if and only if it is conjugate to an element in $\CP_K \dot \t$ for some 

%We define the {\it affine Springer fiber} in the partial affine flag variety $\breve G/\CP_K$ by $$Y_{K, \t}(\g)=\{g \CP_K \in \breve G/\CP_K; g \i \g g \in \CP_K \dot \t\}.$$ By definition, $Y_{K, \t}(\g) \neq \emptyset$ if and only if $\g$ is conjugate to an element in $\CP_K \dot \t$. If we assume furthermore that $\g \in \CP_K \dot \t$, then the condition that $g \i \g g \in \CP_K \dot \t$ is equivalent to the condition that $\g (g \CP_K) \g \i=g \CP_K$. In other words, $Y_{K, \t}(\g)=(\breve G/\CP_K)^\g$ is the subscheme of fixed points of $\g$ on $\breve G/\CP_K$. 

Let $\t \in \Omega^{\ex}$ and $K$ be a $\Ad(\t)$-stable spherical subset $K$ of $\tilde \BS$. Let $\bar \CP_K$ be the reductive quotient of $\CP_K$, i.e., the quotient of $\CP_K$ by its pro-unipotent radical. The Dynkin diagram of $\bar \CP_K$ is $K$. The conjugation action of $\dot \t$ on $\CP_K$ induces a diagram automorphism on $\bar \CP_K$, which we denote by $\bar \t$. 

We define the {\it affine Springer fiber} in the partial affine flag variety $\breve G/\CP_K$ by $$Y_{K, \t}(\g)=\{g \CP_K \in \breve G/\CP_K; g \i \g g \in \CP_K \dot \t\}.$$ By definition, $Y_{K, \t}(\g) \neq \emptyset$ if and only if $\g$ is conjugate to an element in $\CP_K \dot \t$. If we assume furthermore that $\g \in \CP_K \dot \t$, then the condition that $g \i \g g \in \CP_K \dot \t$ is equivalent to the condition that $\g (g \CP_K) \g \i=g \CP_K$. In other words, $Y_{K, \t}(\g)=(\breve G/\CP_K)^\g$ is the subscheme of fixed points of $\g$ on $\breve G/\CP_K$. 

For any element $\g'$ in $\CP_K \dot \t$, we denote by $\bar \g'$ its image in the (not necessarily connected) reductive group $\bar \CP_K \rtimes \<\bar \t\>$. We consider the {\it regular locus} $$Y_{K, \t}^{\reg}(\g)=\{g \CP_K \in Y_{K, \t}(\g); \overline{g \i \g g} \text{ is regular in } \bar \CP_K \rtimes \<\bar \t\>\}.$$ 

Next we recall the notion of elliptic element. For $w \in W_K$, let $\supp(w)$ be the set of simple reflections of that appear in some (or, equivalently, any) reduced expression of $w$. We say that $w$ is a $\t$-elliptic element of $W_K$ if there is no proper $\t$-stable subset of $K$ that contains $\supp(w)$. In other words, $w$ is a $\t$-elliptic element of $W_K$ if and only if $w \notin W_{K'}$ for any $K' \subsetneq K$ with $\Ad(\t)(K')=K'$.  

%Since the set of regular elements in any (not necessarily connected) reductive group over $\kk$ is open, $Y^{\reg}_{K, \t}(\g)$ is open in $Y_{K, \t}(\g)$. %In general, $Y^{\reg}_{K, \t}(\g)$ may not be dense in $Y_{K, \t}(\g)$. 

We have the following result.

%\subsection{Irreducibility for the base case} %\remind{definition of elliptic elements}
\begin{proposition}\label{prop:irr-basic}
Let $\t \in \Omega^{\ex}$ and $K$ be a $\Ad(\t)$-stable spherical subset of $\tilde \BS$. Suppose that $\g \in \CP_K \dot \t$ is a regular semisimple element of $\breve G^{\ex}$. Then for any $\t$-elliptic element $w$ in $W_K$, 
$$\sharp\bigl(Z_{\breve G}(\g) \backslash \Sigma^{{\rm top}}(Y_{w\t}(\g))\bigr)=\sharp\bigl(Z_{\breve G}(\g) \backslash \Sigma^{{\rm top}}(Y^{\reg}_{K, \t}(b))\bigr).$$
\end{proposition}

\begin{proof}
Let $\pi: \breve G/\CI \to \breve G/\CP$ be the projection map. Let $Y'_{K, \t}(\g)$ be the complement of $Y_{K, \t}^{\reg}(\g)$ in $Y_{K, \t}(\g)$ and $Y'_{w\t}(\g)=\pi \i(Y'_{K, \t}(\g)) \cap Y_{w\t}(\g)$.

We consider the following Cartesian product 
\[
\xymatrix{
Y_{w \t}^{\reg}(\g) \ar[r] \ar[d]_{\pi} & Y_{w \t}(\g) \ar[d]^{\pi} & Y'_{w \t}(\g) \ar[l] \ar[d]^{\pi} \\
Y_{K, \t}^{\reg}(\g) \ar[r] & Y_{K, \t}(\g) & Y'_{K, \t}(\g) \ar[l].
}
\] 

Let $\bar B$ be the image of $\CI$ in $\bar \CP_K$. Then $\bar B$ is a Borel subgroup of $\bar \CP_K$. Let $\d'=\Ad(\dot \t)$ be a group automorphism on $\bar \CP_K$. Then $\d'(\bar B)=\bar B$. For $g \CP_K \in Y_{K, \t}(\g)$, we denote by $\bar \g_g$ the image of $g \i \g g \dot \t \i$ in $\bar \CP_K$. It is easy to see that the fiber $\pi \i(g \CP_K) \cap Y_{w \t}(\g)$ is isomorphic to the Lusztig variety $\{h \bar B \in \bar \CP_K/\bar B; h \i \bar \g_g \d'(h) \in \bar B w \bar B\}$. By \cite[Corollary 5.7]{Lu11b}, the fibers are of dimension $\ell(w)$. Hence $\dim Y_{w \t}(\g)=\dim Y_{K, \t}(\g)+\ell(w)$. Moreover, we have $\dim Y^{\reg}_{w \t}(\g)=\dim Y^{\reg}_{K, \t}(\g)+\ell(w)$ and $\dim Y'_{w \t}(\g)=\dim Y'_{K, \t}(\g)+\ell(w)$. 

For any $g \CP_K \in Y'_{K, \t}(\g)$, $\bar \g_g$ is not a $\d'$-regular element in $\bar K$. Hence the fiber $\pi \i(g \CP_K) \cap Y_\g$ is isomorphic to the Springer fiber $\{h \bar B \in \bar K/\bar B; h \i \bar \g_g \d'(h) \in \bar B\}$, and its dimension is larger than $0$. Let $Y'_\g=\pi \i (Y'_{K, \t}(\g)) \cap Y_\g$. Then $\dim Y'_{K, \t}(\g)<\dim Y'_\g \le \dim Y_\t(\g)$. By \cite[Theorem 3.1]{He-ALV}, we have $\dim Y_\g=\dim Y_{K, \t}(\g)$. Hence $\dim Y'_{K, \t}(\g)<\dim Y_{K, \t}(\g)$ and $\dim Y'_{w \t}(\g)<\dim Y_{K, \t}(\g)+\ell(w)=\dim Y_{w \t}(\g)$. Thus the embedding $Y^{\reg}_{K, \t}(\g) \to Y_{K, \t}(\g)$ induces a natural bijection between $\Sigma^{{\rm top}}(Y_{K, \t}^{\reg}(\g))$ and $\Sigma^{{\rm top}}(Y_{K, \t}(\g))$ and 
the embedding $Y_{w \t}^{\reg}(\g) \to Y_{w \t}(\g)$ induces a natural bijection between $\Sigma^{{\rm top}}(Y_{w \t}^{\reg}(\g))$ and $\Sigma^{{\rm top}}(Y_{w \t}(\g))$.

Note that the fibers of the map $\pi: Y_{w \t}^{\reg}(\g) \to Y^{\reg}_{K, \t}(\g)$ are isomorphic to the Lusztig varieties associated with regular elements in $\bar \CP_K$ and elliptic elements in $W_0$. By \cite{HL}, the fibers of $Y_{w \t}^{\reg}(\g) \to Y_{K, \t}^{\reg}(\g)$ are all irreducible. Hence the projection map $\pi$ induces a natural bijection between $\Sigma^{{\rm top}}(Y_{w \t}^{\reg}(\g))$ and $\Sigma^{{\rm top}}(Y^{\reg}_{K, \t}(\g))$.

Note that the embedding maps and the projection map involved here are all equivariant under the action of $Z_{\breve G}(\g)$. Hence the embedding maps and the projection map induces a $Z_{\breve G}(\g)$-equivariant natural bijection between $\Sigma^{{\rm top}}(Y_{w \t}(\g))$ and $\Sigma^{{\rm top}}(Y_{K, \t}(\g))$. 
\end{proof}

\subsection{Enumeration of irreducible components of affine Lusztig varieties}\label{sec:levi}

To enumerate the irreducible components of affine Lusztig varieties, we introduce a subgroup  of $\breve G$ associated with an element $w$ of minimal length in its conjugacy class. 

Recall that $\tW_{\nu_w}$ is the Iwahori-Weyl group of $\breve M_{\nu_w}$ (see \S \ref{sec:HN}). We consider the length function on $\tW_{\nu_w}$ determined by $\CI_{\bM_{\nu_w}}$. This is the same as the length function determined by the positive (relative) root system of $\bM_{\nu_w}$ obtained from the intersection of the positive (relative) root system of $\bG$ with the set of roots of $\bM_{\nu_w}$. Let $\tilde \BS_{\nu_w}$ be the set of simple reflections of $\tW_{\nu_w}$. It is worth pointing out that the length function on $\tW_w$ is not the restriction of the length function on $\tW$ to $\tW_{\nu_w}$ and $\tilde \BS_{\nu_w}$ is not a subset of $\tilde \BS$. 

By \cite[\S 4.2]{He-ALV}, $w=u \t$, where $\t$ is a length-zero element $\t$ of $\tW_w$, and $\supp(u)$ is contained in a $\t$-stable spherical subset of $\BS_{\nu_w}$. Let $K_w$ is the minimal $\t$-stable spherical subset of $\BS_w$ such that $u \in W_{K_w}$ and $\CP_w$ be the standard parahoric subgroup of $\breve M_{\nu_w}$ associated with $K_w$. Then $u$ is a $\t$-elliptic element of $W_{K_w}$ and $\CP_w$ is stable under the conjugation action of $\dot \t$. 

If $\underline p$ is a reduction path, and $w=\text{end}(\underline p)$, we also write $\bM_{\underline p}$ for $\bM_{\nu_w}$, $\CP_{\underline p}$ for $\CP_w$, etc.

Let $\g \in \breve G^{\rs}$ and let $\g' \in \{\g\} \cap \breve M_{\underline p}$. By Theorem \ref{thm:HN1}, we have $Z_{\breve M_{\underline p}}(\g')=Z_{\breve G}(\g')$ and $Y^{\bM_{\underline p}}_{w}(\g') \cong Y_{w}(\g')$. Hence by Proposition \ref{prop:irr-basic}, we have \begin{align*}\tag{a} \sharp\bigl(Z_{\breve G}(\g) \backslash \Sigma^{{\rm top}}(Y_{w}(\g))\bigr) &=\sharp\bigl(Z_{\breve G}(\g') \backslash \Sigma^{{\rm top}}(Y_{w}(\g'))\bigr) \\ &=\sharp\bigl(Z_{\breve M_{\underline p}}(\g') \backslash \Sigma^{{\rm top}}(Y^{\BM_{\underline p}}_{u \t}(\g'))\bigr) \\ &=\sharp\bigl(Z_{\breve M}(\g') \backslash \Sigma^{{\rm top}}(Y^{\BM_{\underline p}}_{K_{\underline p}, \t}(\g'))\bigr).\end{align*}

In particular, $\sharp\bigl(Z_{\breve M_{\underline p}}(\g') \backslash \Sigma^{{\rm top}}(Y^{\BM_{\underline p}}_{K_{\underline p}, \t}(\g'))\bigr)$ is independent of the choice of $\g' \in \{\g\} \cap \breve M_{\underline p}$. We set $$n_{\underline p, \g}=\sharp Z_{\breve G}(\g')\backslash \Sigma^{{\rm top}}(Y^{\bM_{\underline p}}_{K_{\underline p}, \t}(\g')) \text{ for any } \g' \in \{\g\} \cap \breve M_{\underline p}.$$ Now we state the main result of this section. 

\begin{theorem}\label{thm:irr-alv}
    Let $w \in \tW$ and $\CT$ be a reduction tree of $\CT$. Then $$\sharp\bigl(Z_{\breve G}(\g) \backslash \Sigma^{{\rm top}}(Y_w(\g))\bigr)=\sum_{\underline p \in \RaP_{\{\g\}}(\CT); \ell(\underline p)+\ell(\text{end}(\underline p))-\<\bar \nu_{\text{end}(\underline p)}, 2 \rho\>+\dim Y_\g=\dim Y_w(\g)} n_{\underline p, \g}. $$
\end{theorem}

\begin{proof}
By Proposition \ref{prop:dec} and Proposition \ref{prop:dec2}, for any reduction path $\underline p \in \RaP_{\{\g\}}(\CT)$, we have 
\begin{gather*}
\dim Y_{\underline p}(\g)=\dim Y_{{\rm end}(\underline p)}(\g)+\ell(\underline p).
\end{gather*}

%By \cite[Theorem 4.8]{He-Ann}, $\dim X_{{\rm end}(\underline p)}(b)=\ell({\rm end}(\underline p))-\<\nu_b, 2 \rho\>$. Since $\dim X_w(b)=d_w(b)$, $\dim X_{\underline p}(b)=\dim X_w(b)$ if and only if $\underline p$ is cordial. Therefore the inclusion map induces a natural $J_b$-equivariant bijection $$\bigsqcup_{\underline p \text{ is a cordial reduction path of } \CT} \Sigma^{\rm top}(X_{\underline p}(b)) \to \Sigma^{\rm top}(X_w(b)).$$ 

By Theorem \ref{thm:le}, we have $\dim Y_{\text{end}(\underline p)}(\g)=\ell(\text{end}(\underline p))-\<\bar \nu_{\text{end}(\underline p)}, 2 \rho\>+\dim Y_\g$. Hence the inclusion map induces a natural $Z_{\breve G}(\g)$-equivariant bijection $$\bigsqcup_{\underline p \in \RaP_{\{\g\}}(\CT); \ell(\underline p)+\ell(\text{end}(\underline p))-\<\bar \nu_{\text{end}(\underline p)}, 2 \rho\>+\dim Y_\g=\dim Y_w(\g)} \Sigma^{\rm top}(Y_{\underline p}(\g)) \to \Sigma^{\rm top}(Y_w(\g)).$$ 

Since the fibers of the $Z_{\breve G}(\g)$-equivariant projection map $Y_{\underline p}(\g) \to Y_{\rm end(\underline p)}(\g)$ are irreducible, we have a natural bijection $$Z_{\breve G}(\g) \backslash \Sigma^{\rm top}(Y_{\underline p}(\g)) \cong Z_{\breve G}(\g) \backslash \Sigma^{\rm top}(Y_{\rm end(\underline p)}(\g)).$$ By \S\ref{sec:levi} (a) and Proposition \ref{prop:irr-basic}, $$\sharp\bigl(Z_{\breve G}(\g) \backslash \Sigma^{\rm top}(Y_{\underline p}(\g))\bigr)=\sharp\bigl(Z_{\breve G}(\g) \backslash \Sigma^{\rm top}(Y_{\rm end(\underline p)}(\g))\bigr)=n_{\underline p, \g}.$$ 

The statement is proved. 
\end{proof}

\section{Explicit descriptions}

\subsection{Very special reduction paths}

Any element $w \in \tW$ can be written in a unique way as $w=x t^\mu y$ for $x, y \in W_0$ and $\mu \in X_*(T)^+_{\G_0}$ such that $t^\mu y \in {}^{\BS_0} \tW$. We say that $w$ is {\it in the antidominant chamber} if $x=w_0$ and we say that $w$ has the {\it regular translation part} if $\mu$ is regular, i.e., $\<\mu, \a\> >0$ for any simple root $\a$.

We focus on the cases where the reduction paths contributes to the maximal possible dimension (see \S\ref{sec:affineS} (a) and (b)). Note that $X_{\underline p}(b) \neq \emptyset$ if and only if $f_{\tW}(\text{end}(\underline p))=f_{\s}([b])$. In this case $\dim X_{\underline p} (b)=\ell(\underline p)+\dim X_{\text{end}(\underline p)}(b)=\ell(\underline p)+\ell(\text{end}(\underline p))-\<\nu_{[b]}, 2 \rho\>$. Thus we have $\ell(\underline p)+\ell(\text{end}(\underline p))-\<\nu_{[b]}, 2 \rho\> \le \dim X_w(b) \le d_w(b)$. We say that a reduction path $\underline p$ is {\it cordial} if for $b \in \breve G'$ with $f_{\tW}({\rm end}(\underline p))=f_{\s}([b])$, we have $d_w(b)=\ell(\underline p)+\ell({\rm end}(\underline p))-\<\nu_{[b]}, 2 \rho\>$. 

Next, we introduce the notion of very special reduction paths, inspired by the concept of very special parahoric subgroups in $p$-adic groups \cite[\S 2.2]{HZZ}. Several equivalent definitions of very special parahoric subgroups are provided in \cite[Proposition 2.2.5]{HZZ}. Perhaps the most natural definition is as follows: a parahoric subgroup of a $p$-adic group is very special if it has the maximal volume among all parahoric subgroups. However, it is not easy to adapt this definition in a purely combinatorial manner. Instead, we will use the following equivalent definition: a parahoric subgroup is very special if it has the maximal log-volume among all parahoric subgroups. The log-volume is defined as the length of the longest element in its Weyl group.

Let $w \in \tW$ be an element of minimal length in its conjugacy class. By \S\ref{sec:levi}, $w=u \t$, where $\t$ is a length-zero element in the Iwahori-Weyl group $\tW_{\nu_w}$ of $\breve M_{\nu_w}$ and $u$ is a $\t$-elliptic element of $W_{K_{w}}$. For any $\t$-stable spherical subset $K$ of $\tilde \BS_{\nu_w}$, we denote by $n_K$ the length of the longest element of $W_K$. Here we use the length function of $\tW_{\nu_w}$. We say that $w$ is {\it very special} if $n_{K_w} \ge n_K$ for any $\t$-stable spherical subset $K$ of $\tilde \BS_{\nu_w}$. 

Let $\bG'$ be the reduction group over $F'$ associated with $\bG$ and $w$ be a minimal length element in the Iwahori-Weyl group $\tW' \cong \tW$ of $\breve G'$. Let $b=\dot \t$. It is easy to see that $J_b=\breve M_{\nu_w}(\breve F')^{\Ad(\dot \t) \circ \s}$. The standard parahoric subgroups of the $p$-adic group $J_b$ are of the form $\CP_K(\breve F') \cap J_b$, where $K$ runs over the $\t$-stable spherical subsets of $\tilde \BS_{\nu_w}$. By \cite[\S 2.2.4]{HZZ}, $n_K$ is the log-volume of $\CP_K(\breve F') \cap J_b$. By \cite[Proposition 2.2.5]{HZZ}, $w$ is very special if and only if $\CP_{K_w}(\breve F') \cap J_b$ is a very special parahoric subgroup of $J_b$. 

If $\underline p$ is a reduction path and $w=\text{end}(\underline p)$, we say that $\underline p$ is very special if $w$ is very special. We will prove the following result on the reduction tree. 

\begin{proposition}\label{prop:special-path}
Let $w \in \tW$. Suppose that $w \in \tW$ has the regular translation part or is in the antidominant chamber. Let $\CT$ be a reduction tree of $w$. Then all the cordial reduction paths in $\CT$ are very special. 
\end{proposition}

%The proof of Proposition \ref{prop:special-path} will be given in \S \ref{sec:proof}. 

%\remind{revise}

\subsection{The affine Grassmannian case}\label{sec:chen-zhu} 
We first prove Proposition \ref{prop:special-path} when $w=w_0 t^\mu$ for $\mu$ dominant. The proof relies on the explicit description of the stabilizer for the $J_b$ acts on the irreducible components of the affine Deligne-Lusztig varieties in the affine Grassmannian. 

We keep the notation in \S \ref{sec:ADLV}. Let $\widehat{\bG}'$ be the Langlands dual of $\bG'$ over the complex number field $\BC$. Let $\widehat{\bT}'$ be the maximal torus dual to the maximal torus $\bT'$ of $\bG'$. Let $\l_b \in X^*(\widehat{\bT}')$ be the ``best integral approximation'' of the Newton point of $b$ in the sense of \cite[Definition 2.1]{HV}. Let $V_\mu$ be the irreducible representation of $\widehat{\bG'}$ with highest weight $\mu$ and $V_\mu(\l_b)$ be the corresponding $\l_b$-weight subspace of $\widehat{\bT}'$. We have the following result. 

\begin{theorem}\label{Chen-Zhu} 
Let $\mu$ be a dominant coweight and $b \in \breve G'$ with $X_\mu(b) \neq \emptyset$. Then 
\begin{enumerate}
\item The number of $J_b$-orbits on $\Sigma^{{\rm top}}(X_\mu(b))$ equals $\dim V_\mu(\l_b)$.  

\item The stabilizer of each element in $\Sigma^{{\rm top}}(X_\mu(b))$ is a very special parahoric subgroup of $J_b$.
\end{enumerate}
\end{theorem}
 
Part (1) was conjectured by M. Chen and X. Zhu and proved independently by R. Zhou and Y. Zhu \cite{ZZ} and Nie \cite{Nie}. Part (2) was conjectured by X. Zhu and M. Rapoport, and proved independently by Nie \cite{Nie}, and R. Zhou, Y. Zhu, and myself \cite{HZZ}. 

Now we prove Proposition \ref{prop:special-path} for $w=w_0 t^\mu$. By \cite[Proposition 2.4.10]{HZZ}, there exists a natural $J_b$-equivariant bijection $\Sigma^{{\rm top}}(X_{w_0 t^\mu}(b)) \cong \Sigma^{{\rm top}}(X_\mu(b))$. Therefore the stabilizer of each element in $\Sigma^{{\rm top}}(X_{w_0 t^\mu}(b))$ is a very special parahoric subgroup of $J_b$. 

Let $\CT$ be a reduction tree of $w_0 t^\mu$ and $\underline p \in \RaP_{[b]}(\CT)$ be cordial. Then $X_{\underline p}(b) \subset X_w(b)$ is locally closed and $\dim X_{\underline p}(b)=\dim X_w(b)$. Moreover, $X_{\underline p}(b)$ is stable under the $J_b$-action. Thus we have a $J_b$-equivariant injection $\Sigma^{{\rm top}}(X_{\underline p}(b)) \to \Sigma^{{\rm top}}(X_w(b))$. Moreover, by Proposition \ref{prop:dec}, the projection map $X_{\underline p}(b) \to X_{{\text end}(\underline p)}(b)$ induces a $J_b$-equivariant bijection $\Sigma^{{\rm top}}(X_{\underline p}(b)) \to \Sigma^{{\rm top}}(X_{{\text end}(\underline p)}(b))$. Therefore the stabilizer of each element in $\Sigma^{{\rm top}}(X_{{\text end}(\underline p)}(b))$ is a very special parahoric subgroup of $J_b$. 

By \cite[Theorem 4.8]{He-Ann}, $J_b$ acts transitively on $\Sigma^{{\rm top}}(X_{{\text end}(\underline p)}(b))$ and the stabilizer of each element in $\Sigma^{{\rm top}}(X_{{\text end}(\underline p)}(b))$ is isomorphic to the parahoric subgroup $\CP_{\underline p}(\breve F') \cap J_b$ of $J_b$. Hence $\CP_{\underline p}(\breve F') \cap J_b$ is a very special parahoric subgroup of $J_b$. Thus by definition, $\underline p$ is very special. This establishes Proposition \ref{prop:special-path} when $w=w_0 t^\mu$. 

\subsection{Construct the reduction paths}

To deduce the desired property in Proposition \ref{prop:special-path} for other $w$, we construct explicitly some reduction paths from $w_0 t^\mu$ to $w$. This is obtained via the ``partial conjugation method'' in \cite{He07}.

\begin{lemma}\label{lem:path1}
Suppose that $x, y \in W_0$ and $\mu \in X_*(T)$ is dominant and regular. Then there exists a sequence $w_0 t^\mu \lup w_1 \lup \cdots \lup x t^\mu y$ of length $\ell(w_0)-\frac{1}{2}(\ell(x)-\ell(y)+\ell(yx))$.
\end{lemma}

\begin{proof}
Since $\mu$ is dominant and regular, we have $\ell(t^\mu z)=\ell(t^\mu)-\ell(z)$ and $t^\mu z \in {}^{\BS_0} \tW$ for any $z \in W_0$. 

We first show that there exists a sequence $w_0 t^\mu \lup w_1 \lup \cdots \lup y x t^\mu$ of length $\ell(w_0)-\ell(yx)$. 

Let $w_0 (y x) \i=s_{i_1} \cdots s_{i_l}$ be a reduced expression. Then we have $\ell(s_{i_{j+1}} \cdots s_{i_l} y x)=l-j+\ell(yx)$ for any $j$. We have $\ell(s_{i_{j+1}} \cdots s_{i_l} y x t^\mu)=l-j+\ell(y x)+\ell(t^\mu)$ and $$\ell(s_{i_{j+2}} \cdots s_{i_l} y x t^\mu s_{i_{j+1}})=l-(j+1)+\ell(y x)+\ell(t^\mu)-1=\ell(s_{i_{j+1}} \cdots s_{i_l} t^\mu)-2.$$ So $s_{i_{j+1}} \cdots s_{i_l} t^\mu \lup s_{i_{j+2}} \cdots s_{i_l} t^\mu$ for any $j$. Thus we obtain a sequence \[\tag{a} w_0 t^\mu=s_{i_1} \cdots s_{i_l} y x t^\mu \lup s_{i_2} \cdots s_{i_l} y x t^\mu \lup \cdots \lup y x t^\mu\] of length $\ell(w_0)-\ell(yx)$.

Let $y=s_{k_1} \cdots s_{k_m}$ be a reduced expression. Then we have $$\ell(s_{k_j} \cdots s_{k_m} x t^\mu s_{k_1} \cdots s_{k_{j-1}})=\ell(s_{k_j} \cdots s_{k_m} x)+\ell(t^\mu)-(j-1)$$ and \begin{align*} \ell(s_{k_{j+1}} \cdots s_{k_m} x t^\mu s_{k_1} \cdots s_{k_{j}}) &=\ell(s_{k_{j+1}} \cdots s_{k_m} x)+\ell(t^\mu)-j \\ & \le \ell(s_{k_j} \cdots s_{k_m} x)+1+\ell(t^\mu)-j \\ &=\ell(s_{k_j} \cdots s_{k_m} x t^\mu s_{k_1} \cdots s_{k_{j-1}}).\end{align*}

So we have a sequence \begin{align*} y x t^\mu &=s_{k_1} \cdots s_{k_m} x t^\mu \to s_{k_2} \cdots s_{k_m} x t^\mu s_{k_1} \to \cdots  \\ & \to x t^\mu s_{k_1} \cdots s_{k_m}=x t^\mu y.\end{align*} In this sequence, each arrow either preserves the lengths of the elements or decrease the length by $2$. Moreover, the number of edges decreasing the lengths by $2$ equals to $\frac{1}{2}(\ell(y x t^\mu)-\ell(x t^\mu y))=\frac{1}{2} (\ell(y x)-\ell(x)+\ell(y))$. By deleting the edges (and the corresponding elements) which preserves the lengths, we obtain a sequence \[\tag{b} y x t^\mu \lup \cdots \lup x t^\mu y\] of length $\frac{1}{2} (\ell(y x)-\ell(x)+\ell(y))$. 

The statement follows by combining the sequence (a) with the sequence (b). 
\end{proof}

\begin{lemma}\label{lem:path2}
Suppose that $y \in W_0$ and $\mu \in X_*(T)^+_{\G_0}$ with $t^\mu y \in {}^{\BS_0} \tW$. Then there exists a sequence $w_0 t^\mu \lup w_1 \lup \cdots \lup w_0 t^\mu y$ of length $\ell(y)$.
\end{lemma}

\begin{proof}
Let $y=s_{i_1} \cdots s_{i_l}$ be a reduced expression. Since $t^\mu y \in {}^{\BS_0} \tW$, we have $\ell(t^\mu y)=\ell(t^\mu)-l$. For any $j$, since $t^\mu y=(t^\mu s_{i_1} \cdots s_{i_j}) s_{i_{j+1} } \cdots s_{i_l}$, we have $$\ell(t^\mu)-l=\ell(t^\mu y) \ge \ell(t^\mu s_{i_1} \cdots s_{i_j})-(l-j) \ge \ell(t^\mu)-j-(l-j)=\ell(t^\mu)-l.$$ Therefore all the inequalities above must be equalities. In particular, $\ell(t^\mu s_{i_1} \cdots s_{i_j})=\ell(t^\mu)-j$. Hence we also have that $t^\mu s_{i_1} \cdots s_{i_j} \in {}^{\BS_0} \tW$. 

We have $\ell(w_0 t^\mu s_{i_1} \cdots s_{i_j})=\ell(w_0)+\ell(t^\mu)-j$ and \begin{align*} \ell(s_{j+1} w_0 t^\mu s_{i_1} \cdots s_{i_j} s_{j+1}) &=\ell(s_{j+1} w_0)+\ell(t^\mu)-(j+1) \\ &=\ell(w_0)-1+\ell(t^\mu)-(j+1) \\ &=\ell(w_0 t^\mu s_{i_1} \cdots s_{i_j})-2.\end{align*} Hence $w_0 t^\mu s_{i_1} \cdots s_{i_j} \lup w_0 t^\mu s_{i_1} \cdots s_{i_{j+1}}$. We obtain a sequence $$w_0 t^\mu \lup w_0 t^\mu s_{i_1} \lup \cdots w_0 t^\mu s_{i_1} \cdots s_{i_l}=w_0 t^\mu y$$ of length $\ell(y)$. 
\end{proof}

\subsection{Proof of Proposition \ref{prop:special-path}}\label{sec:proof}
We first consider the case where $w=x t^\mu y$ with $\mu$ regular. By Lemma \ref{lem:path1}, there exists a sequence $w_0 t^\mu \lup w_1 \lup \cdots \lup w$ of length $\ell(w_0)-\frac{1}{2}(\ell(x)-\ell(y)+\ell(yx))$. By definition, there exists a reduction tree $\CT'$ of $w_0 t^\mu$ which contains the sequence $w_0 t^\mu \lup w_1 \lup \cdots \lup w$ followed by the reduction tree $\CT$ as a subtree of $\CT'$. For any reduction path  $\underline p: w \lup x_1 \lup \cdots \lup x_m$in $\CT$, we set $\underline p': w_0 t^\mu \lup w_1 \lup \cdots \lup w \lup x_1 \lup \cdots \lup x_m$. Then $\underline p'$ is a reduction path in $\CT'$. Note that $d_{w_0 t^\mu}(b)-d_w(b)=\ell(w_0)-\frac{1}{2}(\ell(x)-\ell(y)+\ell(yx))$. If $\underline p$ is a cordial reduction path in $\CT$, then $\underline p'$ is a cordial reduction path in $\CT'$.  By \S \ref{sec:chen-zhu}, $\underline p'$ is a very special reduction path in $\CT'$. As the definition of the reduction paths only depends on the end point, $\underline p$ is again a very special reduction path in $\CT$. This finishes the proof of the statement when $w$ has the regular translation part. 

When $w$ is in the antidominant chamber, similar argument works by using Lemma \ref{lem:path2} instead of Lemma \ref{lem:path1}.

%Finally, Theorem \ref{thm:main} follows from Proposition \ref{prop:cordialpath} and Proposition \ref{prop:special-path}. 

\subsection{Comparison between the irreducible components}

In this section, we study the cases where the irreducible components of the affine Lusztig varieties and the affine Deligne-Lusztig varieties ``match''. To do this, we need the following deep result on the regular locus of the affine Springer fibers. 

\begin{theorem}\label{thm:ngo}
    Suppose that $\bG$ is split over $L$ and $\g \in \CI$ is a regular semisimple element of $\breve G$. Then $Z_{\breve G}(\g)$ acts transitively on the regular locus $Y^{\reg}_{\BS_0}(\g)$ of the affine Springer fiber in the affine Grassmannian. 
\end{theorem}

This was first proved by Ngo \cite[Proposition 3.7.1]{Ngo} in the context of equal characteristic for Lie algebras. For loop groups in equal characteristic, Chi obtained the result by reducing it to the Lie algebra case, as shown in \cite[Corollary 3.8.2]{Chi}. In the setting of mixed characteristic for loop groups, the result was obtained by Chi in \cite[Corollary 7.6]{Chi2}. However, similar results do not hold for ramified groups, as demonstrated by \cite[Example 7.7 \& Example 7.8]{Chi2}.

Note that any element in $\breve{G}$ is conjugate to an element in $\breve{I} \dot{w} \breve{I}$ for a straight element $w \in \tilde{W}$. We say that $\gamma$ {\it has an integral Newton point} if $\gamma$ is conjugate to an element in $\breve{I} \dot{w} \breve{I}$ for $w = t^\mu$ for some $\mu \in X_*(T)_{\Gamma_0}$. In this case, $f_{\bG, \bG'}({\gamma}) = [t^\mu]$. In particular, all elements in $\CI$ have integral Newton points. It is also worth noting that many elements with integral Newton points are ramified elements.

Now we prove the main result of this section. 

\begin{theorem}\label{thm:main}
Let $\bG$ be a split group over $L$ and $\g$ be a regular semisimple element of $\breve G$ with integral Newton point. Suppose that $\bG'$ is associated with $\bG$ and $[b]=f_{\bG, \bG'}(\{\g\}) \in B_\s(\bG')$. Let $w \in \tW$ with $\dim X_w(b)=d_w(b)$. If $w$ is in the antidominant chamber or has the regular translation part, then $$\sharp\bigl(Z_{\breve G}(\g) \backslash \Sigma^{{\rm top}}(Y_w(\g))\bigr)=\sharp\bigl(J_b \backslash \Sigma^{{\rm top}}(X_w(b))\bigr).$$
\end{theorem}

\begin{proof}
    By definition, $[b]=[t^\mu]$ for some $\mu \in X_*(T)_{\G_0}$. As $[t^\mu]=[t^{x(\mu)}]$ for any $x \in W_0$, we can also assume that $\mu$ is dominant. Let $\bM$ be the Levi subgroup generated by $\bT$ and $\mathbf U_\a$ for $\<\a, \mu\>=0$. Then $\bM$ is standard Levi subgroup of $\bG$ and $\bM$ splits over $L$. In this case, the very special parahoric subgroups of $\breve M$ are the hyperspecial subgroups. 
    
    By definition, we have $\RaP_{[b]}(\CT)=\RaP_{\{\g\}}(\CT)$. Moreover, by Theorem \ref{thm:le}, we have $\dim Y_w(\g)=d_w(b)+\dim Y_\g$. Then for $\underline p \in \RaP_{\{\g\}}(\CT)=\RaP_{[b]}(\CT)$, $\ell(\underline p)+\ell(\text{end}(\underline p))-\<\bar \nu_{\text{end}(\underline p)}, 2 \rho\>+\dim Y_\g=\dim Y_w(\g)$ if and only if $\ell(\underline p)+\ell(\text{end}(\underline p))-\<\bar \nu_{\text{end}(\underline p)}, 2 \rho\>=\dim X_w(b)$, that is, $\underline p \in \RaP_{[b]}(\CT)$ is cordial. 
    
    By Theorem \ref{thm:ngo} for $\bM$ and Proposition \ref{prop:special-path}, $n_{\underline p, \g}=1$ for any cordial reduction path in $\RaP_{\{\g\}}(\CT)$. Now the statement follows from Theorem \ref{thm:irr-alv} and Proposition \ref{prop:irr-adlv}.
\end{proof}

\subsection{Affine Lusztig varieties in the affine Grassmannian}\label{sec:chi-conj}

%Let $\g$ be a regular semisimple element of $\breve G$ and $\mu$ be a dominant coweight with $Y_\mu(\g) \neq \emptyset$. In this subsection, we give an explicit description on the number of $Z_{\breve G}(\g)$-orbits on $\Sigma^{{\rm top}}(Y_\mu(\g))$. 

Let $\widehat{\bG}$ be the Langlands dual of $\bG$ over the complex number field $\BC$. Let $\widehat{\bT}$ be the maximal torus dual to the maximal torus $\bT$ of $\bG$. Suppose that the group $\bG'$ over $F'$ is associated with the group $\bG$ and $[b]=f_{\bG, \bG'}(\{\g\})$. Then by definition, we may identify the Newton point of $\g$ with the Newton point of $b$. In particular, if moreover $\g$ has the integral Newton point $\l_\g$, then under this identification, we have $\l_\g=\l_b$. We have the following result, which verifies a conjecture of Chi \cite{Chi} when $\g$ has an integral Newton point. 

\begin{theorem}
Let $\bG$ be a split group over $L$ and $\g$ be a regular semisimple element of $\breve G$ with an integral Newton point $\l_\g$. Let $\mu$ be a dominant coweight with $Y_{\mu}(\g) \neq \emptyset$. Then $$\sharp\bigl(Z_{\breve G}(\g) \backslash \Sigma^{\rm top}(Y_\mu(\g))\bigr)=\dim V_\mu(\l_\g).$$
\end{theorem}

\begin{proof}
Set $Y=\bigsqcup_{w \in W_0 t^\mu W_0} Y_w(\g)$. By \cite[Proof of Theorem 6.9]{He-ALV}, $Y$ is a $Z_{\breve G}(\g)$-equivariant fibration over $Y_\mu(\g)$ with fibers isomorphic to $\CK^{\mathrm{sp}}/\CI$. In particular, we have a $Z_{\breve G}(\g)$-equivariant bijection $\Sigma^{\rm top}(Y) \to \Sigma^{\rm top}(Y_{\mu}(\g))$. 

Let $\bG'$ be a group over $F'$ that is associated with the group $\bG$ and $[b]=f_{\bG, \bG'}(\{\g\})$. By Theorem \ref{thm:le}, for any $w$, we have either $X_w(b)=Y_w(\g)=\emptyset$ or $\dim Y_w(\g)=\dim X_w(b)+\dim Y_\g$. By \cite[Theorem 9.1]{He-Ann}, for $w \in W_0 t^\mu W_0$ and $w \neq w_0 t^\mu$, we have either $X_w(b)=\emptyset$ or $\dim X_w(b)<\dim X_{w_0 t^\mu}(b)$. Thus for $w \in W_0 t^\mu W_0$ and $w \neq w_0 t^\mu$, we have either $Y_w(\g)=\emptyset$ or $\dim Y_w(\g)<\dim Y_{w_0 t^\mu}(\g)$. The inclusion $Y_{w_0 t^\mu}(\g) \to Y$ induces a $Z_{\breve G}(\g)$-equivariant bijection $\Sigma^{\rm top}(Y_{w_0 t^\mu}(\g)) \to \Sigma^{\rm top}(Y)$. By \cite[Proposition 2.4.10]{HZZ}, we have a $J_b$-equivariant bijection $\Sigma^{\rm top}(X_{w_0 t^\mu}(b)) \to \Sigma^{\rm top}(X_\mu(b))$. By Theorem \ref{thm:main} and Theorem \ref{Chen-Zhu} (1), we have 
\begin{align*}
    \sharp\bigl(Z_{\breve G}(\g) \backslash \Sigma^{\rm top}(Y_\mu(\g))\bigr) &=\sharp\bigl(Z_{\breve G}(\g) \backslash \Sigma^{\rm top}(Y)\bigr)=\sharp\bigl(Z_{\breve G}(\g) \backslash \Sigma^{\rm top}(Y_{w_0 t^\mu}(\g))\bigr) \\ &=\sharp\bigl(J_b \backslash \Sigma^{\rm top}(X_{w_0 t^\mu}(b))\bigr)=\sharp\bigl(J_b \backslash \Sigma^{\rm top}(X_\mu(b))\bigr) \\ &=\dim V_\mu(\l_\g).
\end{align*}

The theorem is proved. 
\end{proof}

As we have seen from the proof, $\sharp\bigl(Z_{\breve G}(\g) \backslash \Sigma^{\rm top}(Y_\mu(\g))\bigr)=\dim V_\mu(\l_\g)$ if and only if $\sharp\bigl(Z_{\breve G}(\g) \backslash \Sigma^{\rm top}(Y_{w_0 \mu}(\g))\bigr)=\sharp\bigl(J_b \backslash \Sigma^{\rm top}(X_{w_0 t^\mu}(b))\bigr)$. By Theorem \ref{thm:irr-alv} and Proposition \ref{prop:irr-adlv}, the above condition is equivalent to the condition that $n_{\underline p, \g}=1$ for all cordial reduction paths in $\RaP_{\{\g\}}(\CT)$ for a reduction tree $\CT$ of $w_0 t^\mu$. However, as we have seen in \cite[Example 7.7 \& Example 7.8]{Chi2}, such condition is not satisfied for the affine Grassmannians of ramified groups. 

\subsection{Irreducibility} The enumeration of the (top dimensional) irreducible components of affine Deligne-Lusztig varieties in the affine flag is a challenging problem. We refer to the recent work of \cite{HNY} and \cite{SSY} for certain cases where $J_b$ acts transitively on $\Sigma^{{\rm top}}(X_w(b))$ and to the work \cite{Sch2} for the description of $J_b \backslash \Sigma^{{\rm top}}(X_w(b))$ in terms of double Bruhat graphs. By Theorem \ref{thm:main}, we may deduce similar results for affine Lusztig varieties. 

%Following \cite[Definition 2.5]{Sch}, for any element $w=x t^\l \in \tW$ with $x \in W_0$ and $\l \in X_*(T)_{\G_0}$ and any (finite) root $\a$, we set $\ell(w, \a)=-\d_{w \a}+\<\l, \a\>+\d_{\a}$, where for a finite root $\b$, $\d(\b)=1$ if $\b$ is positive and $0$ otherwise. By \cite[Definition 2.7]{Sch}, an element $v \in W_0$ is called {\it a length positive element} of $w$ if $\ell(w, v\a) \ge 0$ for all positive finite root $\a$. Following \cite[Definition 3.4]{SSY}, $w \in \tilde W$ is {\it of positive Coxeter type} if there exists a length positive element $v$ of $w$ such that $v \i x v$ is a Coxeter element in a standard parabolic subgroup of $W_0$ (i.e., each simple reflection in $\BS_0$ occurs in a reduced expression of $v \i x v$ at most once). 

%It is discovered by Schremmer, Shimada, and Yu that the affine Deligne-Lusztig variety $X_w(b)$ associated with positive Coxeter type element $w$ has remarkable properties: there is an explicit dimension formula, the group $J_b$ acts transitively on the set of irreducible components of $X_w(b)$, and $X_w(b)$ is an iterated fibration over a classical Deligne-Lusztig variety with $\bG_m$ or $\bA^1$ fibers. In particular, if $w$ is of positive Coxeter type, then $\sharp\bigl(J_b \backslash \Sigma^{{\rm top}}(X_w(b))\bigr)=1$.  

\end{document}